\documentclass[10pt]{amsart}
\usepackage{amssymb,amsmath,mathtools}
\usepackage{enumerate,graphicx}
\usepackage{tikz,float}
\usepackage[tableposition=top]{caption}
\usepackage{hyperref}
\hypersetup{
    colorlinks=true,
    linkcolor=blue,
    filecolor=magenta,
    urlcolor=cyan,
    citecolor=red}

\usetikzlibrary{arrows.meta}
\definecolor{sectorstationary}{HTML}{267C62}
\definecolor{sectornonstationary}{HTML}{B77516}
\definecolor{sectorremote}{HTML}{396EAF}

\numberwithin{equation}{section}
\theoremstyle{plain}
\newtheorem{theorem}{Theorem}[section]
\newtheorem{proposition}[theorem]{Proposition}
\newtheorem{lemma}[theorem]{Lemma}

\theoremstyle{definition}
\newtheorem{remark}[theorem]{Remark}

\newcommand{\R}{\mathbb R}
\newcommand{\Z}{\mathbb Z}
\newcommand{\N}{\mathbb N}
\newcommand{\supp}{\operatorname{supp}}
\newcommand{\diag}{\operatorname{diag}}
\newcommand{\loc}{\mathrm{loc}}
\newcommand{\dd}{\,\mathrm d}

\DeclareMathOperator{\sgn}{sgn}
\allowdisplaybreaks[2]

\begin{document}

\title[The strong $\ell^r$ spherical maximal function]
{The strong $\ell^r$ spherical maximal function}
\author[J.~B. Lee]{Jin Bong Lee}
\author[S. Lee]{Sanghyuk Lee}
\author[J. Oh]{Jeongtae Oh}
\date{}

\begin{abstract}
Let $d\ge3$ and $2\le r<\infty$. We consider the strong maximal
operator associated with Euclidean surface averages over the
$\ell^r$ unit sphere and independent coordinate dilations.
We prove that this operator is bounded on $L^p(\mathbb R^d)$
if and only if
\[
 p>\max\left\{\frac{d+1}{d-1},\frac{r}{d-1}\right\}.
\]
We also obtain $L^p\to L^q$ bounds for the corresponding local
maximal operator. For each fixed $p$ in certain ranges, the resulting
range of $q$ is optimal up to endpoints.
The proof uses a dyadic surface decomposition and
frequency-localized $L^2$ estimates. We establish these estimates
by combining exact multiplier factorization via iterated stationary
phase with scalar Fourier inversion and $TT^*$ arguments.
In dimension three,
these estimates are combined with quantitative multiparameter
local smoothing.
\end{abstract}
\maketitle

\section{Introduction}

Let $d\ge2$, $n=d-1$, $2\le r<\infty$, and
\[
 r'=\frac{r}{r-1},\qquad
 S_r=\left\{y\in\R^d:\sum_{j=1}^d|y_j|^r=1\right\},
 \qquad
 \mathrm D_t=\diag(t_1,\ldots,t_d).
\]
Let $\sigma_r$ denote the Euclidean surface measure on $S_r$. Define
\[
 A f(x,t)
 =
 \int_{S_r} f(x-\mathrm D_ty)\,d\sigma_r(y),
 \qquad t\in(0,\infty)^d,
\]
and define the global maximal operator by
\[
 Mf(x)=\sup_{t\in(0,\infty)^d}|Af(x,t)|.
\]

Our first main result determines the sharp $L^p$ range for the global
maximal operator.

\begin{theorem}\label{thm:global-Lp}
Let $d\ge3$ and $2\le r<\infty$. For $1\le p\le\infty$,
$M$ is bounded on $L^p(\R^d)$ if and only if
\[
 p>p_c(d,r)
 :=
 \max\left\{
 \frac{d+1}{d-1},
 \frac{r}{d-1}
 \right\}.
\]
\end{theorem}
The necessity, including failure at the critical exponents, is proved
in Section~\ref{sec:global-necessity}.

For $r=2$, $M$ is the strong spherical maximal operator.
Its one-parameter counterpart is the classical spherical maximal
operator, whose sharp strong-type range is $p>d/(d-1)$ for
$d\ge3$, by Stein \cite{Stein76}, and $p>2$ for $d=2$, by
Bourgain \cite{Bo1986}. The adjective ``strong'' refers to
independent coordinate dilations, in analogy with strong
differentiation of integrals; see \cite{JMZ1935}.

An early study of a related planar operator is due to
Erdo\u{g}an \cite{Erdogan2003}, who obtained Sobolev estimates
for a local maximal operator over rotated ellipses.
For axis-parallel ellipses, Lee, Lee, and Oh \cite{LLO2025-1}
proved that the global strong circular maximal operator is
bounded on $L^p$ for $p>4$, and that boundedness fails for
$p\le3$. Chen, Guo, and Yang \cite{CGY} subsequently obtained
the sharp local range $p>3$ by combining multi-parameter
local smoothing with the incidence estimates of Pramanik,
Yang, and Zahl \cite{PYZ}; see
\cite[Corollary~1.2 and Section~3]{CGY}.
The resulting frequency-localized bounds, combined with the
local-to-global argument in \cite[Section~2.1]{LLO2025-1},
yield the sharp global strong circular range $p>3$.

In dimensions $d\ge3$, Lee, Lee, and Oh \cite{LLO2025-2}
established boundedness of the strong spherical maximal
operator for $p>2(d+1)/(d-1)$. Hickman and Zahl \cite{HZ2025}
improved this range to $p>2$, which is sharp when $d=3$.
They also proved failure of boundedness for
$p\le(d+1)/(d-1)$, including the critical exponent; see
\cite[Appendix~A]{HZ2025}.
Theorem~\ref{thm:global-Lp} determines the sharp global range
for the full family of $\ell^r$ spheres considered here.

To prove Theorem~\ref{thm:global-Lp}, we first obtain quantitative
frequency-localized estimates for the local maximal operator
\[
 M_{\loc}f(x)
 =
 \sup_{t\in[1,2]^d}|Af(x,t)|.
\]
Local spherical maximal operators admit $L^p$-improving estimates;
see \cite{ScSo1997, Lee2003} for the one-parameter circular case.
In addition to the diagonal estimate needed for Theorem~\ref{thm:global-Lp}, our argument gives the following $L^p$-improving range.
\begin{theorem}\label{thm:lplq}
Let $d\ge4$, $2\le r<\infty$, and $1\le p,q<\infty$. Then
\[
 \|M_{\loc}f\|_{L^q(\R^d)}
 \lesssim_{d,r,p,q}
 \|f\|_{L^p(\R^d)}
\]
whenever
\[
 p>\frac{d+1}{d-1},
 \qquad
 p\le q<dp,
\]
and
\[
 \frac{d-1}{p}-\frac{d(d-3)}{q}<1,
 \qquad
 \frac{r+1}{p}-\frac{r(d-2)+1}{q}<1,
 \qquad
 \frac{r+d-1}{p}-\frac{d-1}{q}<d-1.
\]
\end{theorem}

In dimension three we obtain the following sufficient improving range.
\begin{theorem}\label{thm:lplq-d3}
Let $d=3$ and $2\le r<\infty$. Then
\[
 \|M_{\loc}f\|_{L^q(\R^3)}
 \lesssim_{p,q,r}\|f\|_{L^p(\R^3)}
\]
whenever
\[
 2<p<\infty,\qquad p\le q\le2p,\qquad
 \frac{r+2}{p}-\frac2q<2.
\]
\end{theorem}
The diagonal specialization is $p>\max\{2,r/2\}$.
Theorem~\ref{thm:lplq-d3} does not determine the full off-diagonal
range. Its proof combines the frequency-localized
$L^2$ estimates with an annular $L^6$ estimate; the local smoothing
argument underlying the latter is given in Appendix~\ref{app:d3-smoothing}.

Our sufficient region agrees with the region allowed by the tests in
Proposition~\ref{prop:global-necessity}, except for the upper boundary, when
$d\ge4$ and $p\ge\max\{2,(2r+d-1)/d\}$, or when
$d\ge3$, $r>3$, and
$p_c(d,r)<p\le(r+d-2)/(d-1)$.
The remaining boundaries are, respectively, $q=dp$ and
$(r+d-1)/p-(d-1)/q=d-1$.

Gaps remain in other parts of the exponent region.
For $r=2$ and $d\ge4$, our sufficient range is smaller than
the range allowed by the necessary conditions when
$p_c(d,2)<p<2$.
For $d=3$ and $r=2$, we obtain $p\le q\le2p$,
whereas the necessary conditions allow $p\le q\le3p$.

For comparison, consider in dimensions $d\ge4$ the local strong
maximal operator associated with a smoothly cut off compact patch of the
shifted quadratic graph $y=(v,c+|v|^2)$, where $c\ne0$.
For its separated quadratic phase, the tangential $TT^*$ estimate
and spatial localization, followed by interpolation and dyadic
summation, give the sufficient range
\[
 p>\frac{d+1}{d-1},\qquad p\le q<(d-1)p;
\]
see Remark~\ref{rem:shifted-quadratic-model}.
This direct summation does not determine the boundary $q=(d-1)p$.
The spherical estimates extend beyond this range through the additional
oscillation described below.

Our proof begins with a finite decomposition of $S_r$ into coordinate
graph patches. When $r=2$, the sphere has nonvanishing Gaussian
curvature, and we can use the stationary phase expansion of the Fourier
transform of surface measure directly. For $r>2$, the curvature
degenerates along the coordinate hyperplanes. We further decompose
each graph patch into dyadic patches within the orthants and obtain
an exact factorization of each frequency-sector multiplier by iterated
stationary phase.

The canonical stationary-phase function for the $\ell^r$ sphere is
the dual norm
\[
 h(\xi)=\left(\sum_{j=1}^d|\xi_j|^{r'}\right)^{1/r'}.
\]
For $r=2$, this is $h(\xi)=|\xi|$, the phase of the classical wave
propagator. Scalar Fourier inversion expresses $e^{-iRh}$ as an
integral of exponentials with separated phases. After linearizing the
parameter supremum, we obtain an oscillatory integral whose nonlinear
phase is a sum of one-dimensional powers of degree $r'$.
These powers are quadratic when $r=2$.

In the Euclidean case, the scalar Fourier step takes the form
\[
 e^{-iR\sqrt z}=\int_{\R}K_R(s)e^{isz}\,ds,
 \qquad z\asymp1,
\]
where a smooth cutoff equal to one on the relevant compact positive
$z$-interval is included in the definition of $K_R$.
This representation is related to the subordination formula in
M\"uller--Seeger~\cite[Lemma~4.3]{MullerSeeger}.
For the $L^2\to L^4$ estimate, an $L^2$ bound in the scalar variable
reduces the problem to a $TT^*$ estimate in the $d-1$ tangential
coordinates. For the $L^2\to L^{2d}$ estimate in dimensions $d\ge4$,
we retain the scalar variable in the full $TT^*$ calculation and combine
it with oscillation in all $d$ spatial coordinates. This use of the
square-root phase gives estimates beyond the range obtained by direct
summation for the quadratic model.

The proof is organized as follows. Section~\ref{sec:coordinate-decay}
transfers vertical frequency decay to every coordinate.
Section~\ref{sec:frequency-localized} states the normalized estimates
and derives the local theorems by interpolation and dyadic summation.
Section~\ref{sec:proof-local-patch} proves these estimates using
factorization and $TT^*$; the iterated stationary-phase argument is
given in Section~\ref{sec:proof-lem-exact}.
Section~\ref{sec:global-theorem} completes the global theorem and
proves the necessary conditions, including failure at the critical
exponents. Appendix~\ref{app:d3-smoothing} supplies the additional
local smoothing input required for the three-dimensional results.

In the Euclidean case $r=2$ and dimensions $d\ge4$, the global
range in Theorem~\ref{thm:global-Lp} coincides with that obtained
by Mingfeng Chen \cite{Chen2026}.
Chen uses Fourier inversion in the defining function of the
sphere and, after squaring one input coordinate, reduces the
local $L^2$ estimate to a quadratic $TT^*$ argument.
Our proof uses exact factorization of the surface-measure
multiplier and scalar Fourier inversion to obtain the
corresponding $TT^*$ estimates.
The present results also cover $\ell^r$ spheres for $r>2$
and provide the local $L^p$-improving estimates stated above.

\section{Coordinatewise frequency decay}
\label{sec:coordinate-decay}

To pass from $M_{\loc}$ to $M$, we need frequency decay in every
coordinate. The following lemma transfers decay in the vertical
frequency $\xi_d$ to each coordinate $\xi_i$. Its proof uses only
one tangential derivative of a kernel regularized in the vertical
variable, without coordinate symmetry or repeated integration
by parts.

Let $d\ge2$, let $U\subset\R^{d-1}$ be open, and suppose that
$a\in C_c^1(U)$ and $\phi\in C^1(U)$. For
$t=(t_1,\ldots,t_d)\in[1,2]^d$, set
\[
 A f(x,t)
 =
 \int_U
 f\bigl(x-\mathrm D_t(y,\phi(y))\bigr)a(y)\dd y,
 \qquad
 M_{\loc}f(x)
 =
 \sup_{t\in[1,2]^d}|A f(x,t)|.
\]
The amplitude $a$ need not be nonnegative. Initially, all operators
are applied to Schwartz functions.

Choose an even function $\psi_0\in C_c^\infty(\R)$ such that
$\psi_0=1$ on $[-1,1]$ and $\supp\psi_0\subset[-2,2]$, and put
$\psi(\eta)=\psi_0(\eta)-\psi_0(2\eta)$. Define
\[
 \widehat{L_0^{(i)}f}(\xi)
 =
 \psi_0(\xi_i)\widehat f(\xi),
 \qquad
 \widehat{P_m^{(i)}f}(\xi)
 =
 \psi(2^{-m}\xi_i)\widehat f(\xi),
 \quad m\ge1.
\]
Thus $P_m^{(i)}$ is supported where
$2^{m-1}\le|\xi_i|\le2^{m+1}$, and
\begin{equation}\label{eq:partition}
 I=L_0^{(i)}+\sum_{m\ge1}P_m^{(i)}.
\end{equation}

\begin{lemma}\label{lem:global-decay}
Fix $1<p<\infty$. Suppose that there are $C_0<\infty$ and
$\epsilon>0$ such that
\begin{equation}\label{eq:assumption}
 \|M_{\loc}P_m^{(d)}f\|_{L^p(\R^d)}
 \le
 C_0 2^{-\epsilon m}\|f\|_{L^p(\R^d)},
 \qquad m\ge1.
\end{equation}
Then, with $\gamma=\epsilon/(2+\epsilon)$, for every
$i\in\{1,\ldots,d\}$ and $k\ge1$,
\begin{equation}\label{eq:conclusion}
 \|M_{\loc}P_k^{(i)}f\|_{L^p(\R^d)}
 \le
 C2^{-\gamma k}\|f\|_{L^p(\R^d)}.
\end{equation}
The constant is independent of $k$ and $f$, and may depend on
$d,p,C_0,\epsilon$, the cutoffs, the support of $a$, and the $C^1$
bounds for $a$ and $\phi$ on a neighborhood of $\supp a$.
\end{lemma}

To prove Lemma~\ref{lem:global-decay}, we need estimates for the
interaction between frequency localizations in two different
coordinates.

\begin{lemma}\label{lem:rough}
For $i<d$ and $k,m\ge1$,
\begin{align}
 \|M_{\loc}P_k^{(i)}P_m^{(d)}f\|_p
 &\le
 C2^{2m-k}\|f\|_p,
 \label{eq:rough}\\
 \|M_{\loc}P_k^{(i)}L_0^{(d)}f\|_p
 &\le
 C2^{-k}\|f\|_p.
 \label{eq:low}
\end{align}
These estimates hold for every $1\le p\le\infty$, without assuming
\eqref{eq:assumption}.
\end{lemma}

\begin{proof}
Choose a compactly supported $C^1$ extension of $\phi$ that agrees with
the original function on a neighborhood of $\supp a$, and extend $a$ by
zero. Such an extension of $\phi$ is obtained by multiplying by a smooth
cutoff supported in $U$ and equal to one near $\supp a$.
We may therefore regard both functions as defined on $\R^{d-1}$.

Write $\kappa=\check\psi$, $\lambda=\check\psi_0$, and
$\kappa_m(u)=2^m\kappa(2^m u)$. A change of variables in the graph
integral shows that 
\begin{align}
	A P_m^{(d)}f (x,t)
	={}&\int_{\mathbb R^{d-1}}\int_{\mathbb R}
	f(x-(y,s))K_{t,m}(y,s)\dd s\dd y,\nonumber\\
	K_{t,m}(y,s)
	={}&\Big(\prod_{i=1}^{d-1} t_i\Big)^{-1}
	a(y_1/t_1,\dots,y_{d-1}/t_{d-1})
	\nonumber\\
	&\times 2^m\check\psi\left(
	2^m\left(
	s-t_d\phi(y_1/t_1,\dots,y_{d-1}/t_{d-1})
	\right)\right).
\end{align}
For notational simplicity, set
\[
	u=(y_1/t_1,\ldots,y_{d-1}/t_{d-1}),
	\qquad
	v=s-t_d\phi(u).
\]

Here the vertical projection acts at the fixed physical frequency $2^m$;
in particular, there is no additional factor $t_d^{-1}$ in this formula.

For $i<d$, we have
\begin{equation}\label{eq:derivative}
	\partial_{y_i}K_{t,m}(y,s)
	=
	\Big(\prod_{j=1}^{d-1} t_j\Big)^{-1}t_i^{-1}
	\left[
	(\partial_i a)(u)\kappa_m(v)
	-t_da(u)(\partial_i\phi)(u)\kappa_m'(v)
	\right].
\end{equation}

We next establish an integrable envelope uniform in $t$.
Choose $c\ge1$ so that $\supp a\subset[-c,c]^{d-1}$, and choose
$C\ge1$ so that $|t_d\phi(u)|\le C$ whenever $u\in\supp a$
and $t_d\in[1,2]$. Since $\kappa$ is a Schwartz function,
\[
 |\kappa_m(v)|+|\kappa_m'(v)|
 \le C_N2^{2m}(1+2^m|v|)^{-N},
 \qquad m\ge1.
\]
If $|s|>2C$, then
$|v|=|s-t_d\phi(u)|\ge|s|/2$ on the relevant support.
If $|s|\le2C$, we simply use the bound $C2^{2m}$.
Since $\supp\partial_i a\subset\supp a$, the equality
\eqref{eq:derivative} consequently implies
\[
	F_m(z):=\sup_{t\in[1,2]^d}|\partial_{y_i}K_{t,m}(y,s)|
	\le
	C_N2^{2m}\mathbf 1_{[-2c,2c]^{d-1}}(y)(1+|s|)^{-N}.
\]
Taking $N>1$ yields
\begin{equation}
 \left\|
 \sup_{t\in[1,2]^d}|\partial_{y_i}K_{t,m}|
 \right\|_{L^1(\R^d)}
 \lesssim 2^{2m}.
\end{equation}

Let $\kappa_k(h)=2^k\kappa(2^kh)$ be the one-dimensional kernel of
$P_k^{(i)}$. Since $\psi(0)=0$,
\[
 \int_\R \kappa_k(h)\dd h=0,
 \qquad
 \int_\R |h|\,|\kappa_k(h)|\dd h
 =
 2^{-k}\int_\R |u|\,|\kappa(u)|\dd u
 \le C2^{-k}.
\]
The kernel $H_{t,k,m}$ of $A_tP_k^{(i)}P_m^{(d)}$ therefore satisfies
\begin{align*}
	H_{t,k,m}(z)
	&=\int_\R
	\kappa_k(h)\bigl(
	K_{t,m}((y,s)-he_i)-K_{t,m}(y,s)
	\bigr)\dd h\\
	&=-\int_\R h\kappa_k(h)
	\int_0^1
	\partial_{y_i}K_{t,m}((y,s)-\theta he_i)
	\dd\theta\dd h.
\end{align*}
By taking the supremum in $t$ and integrating in $(y,s)$, we obtain
\begin{align*}
	\left\|\sup_t|H_{t,k,m}|\right\|_1
	&\le
	\int_\R |h|\,|\kappa_k(h)|
	\int_0^1
	\left\|
	\sup_{t\in[1,2]^d}
	|\partial_{y_i}K_{t,m}(\,\cdot-\theta he_i)|
	\right\|_1
	\dd\theta\dd h
	\notag\\
	&\le C2^{2m-k}.
\end{align*}
Pointwise,
\[
 M_\loc P_k^{(i)}P_m^{(d)}f(x)
 \le
 \bigl(|f|*\sup_t|H_{t,k,m}|\bigr)(x).
\]
Young's inequality proves \eqref{eq:rough} for all
$1\le p\le\infty$.

For \eqref{eq:low}, replace $\kappa_m$ in the argument above by
$\lambda=\check\psi_0$. The same argument gives
\[
	\left\|
	\sup_{t\in[1,2]^d}|\partial_{y_i}K_{t,0}|
	\right\|_1
	\le C,
\]
and hence the bound $C2^{-k}$.
\end{proof}

Now we prove Lemma~\ref{lem:global-decay}.

\begin{proof}[Proof of Lemma~\ref{lem:global-decay}]
For $i=d$, the conclusion follows directly from \eqref{eq:assumption},
since $\gamma<\epsilon$. Fix $i<d$. The coordinate multipliers commute,
and $\|P_k^{(i)}\|_{p\to p}\le\|\kappa\|_1$ uniformly in $k$.
Applying \eqref{eq:assumption} to $P_k^{(i)}f$ gives
\begin{equation}
 \|M_\loc P_k^{(i)}P_m^{(d)}f\|_p
 \le C2^{-\epsilon m}\|f\|_p.
\end{equation}
Combining this with Lemma~\ref{lem:rough}, we obtain
\[
 \|M_\loc P_k^{(i)}P_m^{(d)}f\|_p
 \le C\min\{2^{2m-k},2^{-\epsilon m}\}\|f\|_p.
\]

Let $L\ge0$ be an integer. Decompose in the $d$-th frequency
coordinate using \eqref{eq:partition}, and use the sublinearity of
$M_\loc$:
\begin{align*}
 \|M_\loc P_k^{(i)}f\|_p
 &\le
 \|M_\loc P_k^{(i)}L_0^{(d)}f\|_p
 +\sum_{1\le m\le L}
 \|M_\loc P_k^{(i)}P_m^{(d)}f\|_p\\
 &\hspace{2cm}
 +\sum_{m>L}
 \|M_\loc P_k^{(i)}P_m^{(d)}f\|_p\\
 &\le
 C\left(
 2^{-k}
 +\sum_{1\le m\le L}2^{2m-k}
 +\sum_{m>L}2^{-\epsilon m}
 \right)\|f\|_p\\
 &\le
 C_\epsilon
 \left(
 2^{-k}+2^{2L-k}+2^{-\epsilon L}
 \right)\|f\|_p.
\end{align*}
Choose
\[
 L=\left\lfloor\frac{k}{2+\epsilon}\right\rfloor,
 \qquad
 \gamma=\frac{\epsilon}{2+\epsilon}.
\]
Then $2L-k\le-\gamma k$, while
$L\ge k/(2+\epsilon)-1$ gives
$-\epsilon L\le-\gamma k+\epsilon$. Since $\gamma<1$,
$2^{-k}\le2^{-\gamma k}$. Hence
\[
 \|M_\loc P_k^{(i)}f\|_p
 \le C_\epsilon2^{-\gamma k}\|f\|_p,
\]
which proves \eqref{eq:conclusion}.
\end{proof}

\section{Frequency-localized estimates}
\label{sec:frequency-localized}

We begin with the dyadic surface decomposition and state the normalized
analytic inputs. We then recover estimates for the original patches.
After a common low-frequency bound, we derive the local theorems
separately in dimensions $d\ge4$ and $d=3$.

\subsection{Dyadic patches and normalized estimates}

We decompose $S_r$ into finitely many coordinate graph patches using
the restrictions of smooth ambient cutoffs. On each patch, one coordinate is bounded away from
zero; after permuting and reflecting the coordinates, the patch may be
written as
\[
 y=(v,\Phi(v)),\qquad
 \Phi(v)=\left(1-\sum_{i=1}^{d-1}|v_i|^r\right)^{1/r},
\]
with $\Phi$ bounded below on the support of the corresponding cutoff.
A patchwise piece of $A$ therefore has the form
\[
 A_{\chi} f(x,t)
 =
 \int_{\R^{d-1}}
 f\left(
 x_1-t_1v_1,\ldots,x_{d-1}-t_{d-1}v_{d-1},
 x_d-t_d\Phi(v)
 \right)\chi(v)\dd v.
\]
Here $\chi$ includes the partition cutoff and the Euclidean surface
density. For general real $r$, the graph parametrization and $\chi$ need not
be smooth across coordinate hyperplanes. The uniform smoothness
needed below is obtained after dyadic rescaling within each orthant.
Since the full maximal operator is pointwise dominated by the sum of
the finitely many patchwise maximal operators, it is enough to prove
the required estimates on each such patch.

Fix one of these graph patches and, throughout this section, denote its
localized averaging operator by $A$. Let $n=d-1$. After a finite
decomposition into orthants, we may work in the positive orthant and
choose $\zeta\in C_c^\infty([1/2,2])$ such that
\[
 \sum_{m\ge0}\zeta(2^m v_i)=1
\]
on the support under consideration. For
$\mathbf m=(m_1,\ldots,m_n)\in\mathbb Z_{\ge0}^n$, write
\[
 A f(x,t)
 =
 \sum_{\mathbf m\in\mathbb Z_{\ge0}^n}A^{\mathbf m}f(x,t),
\]
where
\[
 A^{\mathbf m}f(x,t)
 =
 \int_{\mathbb R^n}
 f\left(x-\mathrm D_t(v,\Phi(v))\right)
 \chi(v)\prod_{i=1}^n\zeta(2^{m_i}v_i)\,dv.
\]

Set
\[
 \delta_i=2^{-m_i},
 \qquad
 V_\delta=\prod_{i=1}^n\delta_i,
 \qquad
 \Phi_\delta(v)
 =
 \left(1-\sum_{i=1}^n(\delta_i v_i)^r\right)^{1/r}.
\]
By the change of variables $v_i=\delta_i w_i$, followed by relabeling
$w$ as $v$, we have
\[
 A^{\mathbf m}f(x,t)
 =
 V_\delta
 \int_{\mathbb R^n}
 f\left(
 x-
 (t_1\delta_1v_1,\ldots,t_n\delta_nv_n,
 t_d\Phi_\delta(v))
 \right)
 B_\delta(v)\,dv,
\]
where
\[
 B_\delta(v)
 =
 \chi(\delta_1v_1,\ldots,\delta_nv_n)
 \prod_{i=1}^n\zeta(v_i).
\]
Thus $B_\delta$ is supported in the fixed box $1/2<v_i<2$, and, by
the choice of the graph patch, $\Phi_\delta$ is bounded below by a
fixed positive constant on a slightly enlarged support of $B_\delta$.

Let
\[
 M_{\loc}^{\mathbf m}f(x)
 =
 \sup_{t\in[1,2]^d}|A^{\mathbf m}f(x,t)|
\]
and set
\[
 \lambda_i=\max\{1,R\delta_i^r\}.
\]
For $R\ge2$, let $P_R^{(d)}$ denote the multiplier
$\psi(\xi_d/R)$. When summing the high-frequency pieces, we take
$R=2^j$, $j\ge1$.

To normalize the horizontal translation scales, define
\[
 A_\delta f(x,t)
 =
 V_\delta\int_{\mathbb R^n}
 f(x-\mathrm D_t(v,\Phi_\delta(v)))B_\delta(v)\,dv,
 \qquad
 M_{\loc}^\delta f(x)
 =
 \sup_{t\in[1,2]^d}|A_\delta f(x,t)|.
\]
Let
\[
 E_\delta=\diag(\delta_1,\ldots,\delta_n,1).
\]
Then
\[
 A^{\mathbf m}f(E_\delta x,t)
 =
 A_\delta(f\circ E_\delta)(x,t),
\]
and hence
\begin{equation}\label{eq:exact-spatial-rescaling}
 M_{\loc}^{\mathbf m}f(E_\delta x)
 =
 M_{\loc}^\delta(f\circ E_\delta)(x).
\end{equation}
Since the last coordinate is unchanged by $E_\delta$,
$P_R^{(d)}$ commutes with this rescaling.

For a positive integer $M$, put
\[
 \mathfrak B_M
 =
 \max_{|\alpha|\le M}
 \|\partial_v^\alpha B_\delta\|_\infty,
 \qquad
 I=\{i:R\delta_i^r\ge1\},
\]
and set
\begin{align}
 b_{\delta,4}(R)
 &=
 R^{1/2}V_\delta
 \prod_{i=1}^n\lambda_i^{-1/4},
 \\
 b_{\delta,2d}(R)
 &=
 R^{(d-1)/d}V_\delta
 \prod_{i=1}^n\lambda_i^{-1/d}.
\end{align}

The main analytic inputs are the following $L^2$-improving
estimates, with their indicated dimension restrictions.

\begin{lemma}\label{lem:normalized-improving}
Let $d\ge3$. There is a finite $M=M(d,r)$ such that, uniformly in
$\delta$ and $R\ge2$,
\begin{equation}\label{eq:normalized-L4-bound}
 \|M_{\loc}^\delta P_R^{(d)}f\|_4
 \lesssim \mathfrak B_M b_{\delta,4}(R)\|f\|_2.
\end{equation}
If $d\ge4$, then
\begin{equation}\label{eq:normalized-L2d-bound}
 \|M_{\loc}^\delta P_R^{(d)}f\|_{2d}
 \lesssim \mathfrak B_M b_{\delta,2d}(R)\|f\|_2.
\end{equation}
If $d=3$, then
\begin{equation}\label{eq:normalized-L12-bound}
 \|M_{\loc}^\delta P_R^{(3)}f\|_{12}
 \lesssim \mathfrak B_M R^{3/4}V_\delta
 (\lambda_1\lambda_2)^{-1/4}\|f\|_2.
\end{equation}
\end{lemma}

Spatial localization in Section~\ref{sec:normalized-L2} also gives
the companion $L^2$ estimate
\begin{equation}\label{eq:normalized-L2-bound}
 \|M_{\loc}^\delta P_R^{(d)}f\|_2
 \lesssim
 \mathfrak B_M b_{\delta,4}(R)\|f\|_2.
\end{equation}

Lemma~\ref{lem:normalized-improving} and the companion estimate are
proved in Section~\ref{sec:proof-local-patch}. They yield the
following bounds for the original dyadic patches.

\begin{proposition}\label{prop:dyadic-patch-bounds}
Let $d\ge3$. Uniformly in $\mathbf m$ and $R\ge2$,
\begin{align}
 \|M_{\loc}^{\mathbf m}P_R^{(d)}f\|_2
 &\lesssim R^{1/2}V_\delta
 \prod_{i=1}^n\lambda_i^{-1/4}\|f\|_2,
 \label{eq:dyadic-22}\\
 \|M_{\loc}^{\mathbf m}P_R^{(d)}f\|_4
 &\lesssim R^{1/2}V_\delta^{3/4}
 \prod_{i=1}^n\lambda_i^{-1/4}\|f\|_2.
 \label{eq:dyadic-24}
\end{align}
If $d\ge4$, we also have
\begin{equation}
 \|M_{\loc}^{\mathbf m}P_R^{(d)}f\|_{2d}
 \lesssim R^{(d-1)/d}V_\delta^{(d+1)/(2d)}
 \prod_{i=1}^n\lambda_i^{-1/d}\|f\|_2.
\end{equation}
For $d=3$, the additional estimate is
\begin{equation}\label{eq:dyadic-2-12}
 \|M_{\loc}^{\mathbf m}P_R^{(3)}f\|_{12}
 \lesssim R^{3/4}V_\delta^{7/12}
 (\lambda_1\lambda_2)^{-1/4}\|f\|_2.
\end{equation}
Moreover, for every $d\ge3$,
\begin{align}
 \|M_{\loc}^{\mathbf m}P_R^{(d)}f\|_\infty
 &\lesssim V_\delta\|f\|_\infty,\\
 \|M_{\loc}^{\mathbf m}P_R^{(d)}f\|_1
 &\lesssim RV_\delta\|f\|_1,\\
 \|M_{\loc}^{\mathbf m}P_R^{(d)}f\|_\infty
 &\lesssim R\|f\|_1.
\end{align}
\end{proposition}

\begin{proof}[Proof of Proposition~\ref{prop:dyadic-patch-bounds}]
Since
\[
 \|f\circ E_\delta\|_2=V_\delta^{-1/2}\|f\|_2,
\]
the rescaling \eqref{eq:exact-spatial-rescaling} gives
\begin{align*}
 \|M_{\loc}^{\mathbf m}P_R^{(d)}f\|_2
 &\lesssim R^{1/2}V_\delta
 \prod_{i=1}^n\lambda_i^{-1/4}\|f\|_2,\\
 \|M_{\loc}^{\mathbf m}P_R^{(d)}f\|_4
 &\lesssim R^{1/2}V_\delta^{3/4}
 \prod_{i=1}^n\lambda_i^{-1/4}\|f\|_2
\end{align*}
for $d\ge3$, and
\[
 \|M_{\loc}^{\mathbf m}P_R^{(d)}f\|_{2d}
 \lesssim R^{(d-1)/d}V_\delta^{(d+1)/(2d)}
 \prod_{i=1}^n\lambda_i^{-1/d}\|f\|_2
 \quad(d\ge4).
\]
For $d=3$, the same change of variables gives
\[
 \|M_{\loc}^{\mathbf m}P_R^{(3)}f\|_{12}
 \lesssim R^{3/4}V_\delta^{7/12}
 (\lambda_1\lambda_2)^{-1/4}\|f\|_2,
\]
since the spatial norm contributes $V_\delta^{1/12-1/2}$.
To justify the uniform amplitude bounds for every real $r\ge2$, let
$\kappa$ be the smooth ambient cutoff defining the graph patch. On the
positive orthant the Euclidean surface density gives
\[
 B_\delta(v)=\kappa(\delta v,\Phi_\delta(v))
 \left(1+\Phi_\delta(v)^{2-2r}
       \sum_{i=1}^n\delta_i^{2r-2}v_i^{2r-2}\right)^{1/2}
 \prod_{i=1}^n\zeta(v_i).
\]
The cutoff is supported where the graph height has a fixed positive
lower bound, and the amplitude is extended by zero away from this
patch. On a fixed enlargement of its support, $v_i$ and the graph
height are bounded away from zero. All fixed-order derivatives of
$v_i^r$ and $v_i^{2r-2}$ are therefore bounded there. Since
$\delta_i^r,\delta_i^{2r-2}\le1$, the chain rule gives, for every
fixed $M$,
\[
 \sup_{\mathbf m\in\mathbb Z_{\ge0}^n}\|B_{2^{-\mathbf m}}\|_{C^M}
 \le C_{M,d,r,\kappa,\zeta}.
\]
Thus $\mathfrak B_M$ may be absorbed into the implicit constants.
The constants may depend on the fixed value of $r$.

The remaining three estimates follow from elementary kernel bounds.

The convolution kernel of
$A^{\mathbf m}P_R^{(d)}$ is
\begin{equation}\label{eq:elementary-patch-kernel}
 \begin{aligned}
 K_{t,R}^{\mathbf m}(z)
 &=
 \left(\prod_{i=1}^nt_i^{-1}\right)
 B_\delta\left(
 \frac{z_1}{t_1\delta_1},\ldots,
 \frac{z_n}{t_n\delta_n}
 \right)\\
 &\quad\times R\check\psi\left(
 R\left[
 z_d-t_d\Phi_\delta\left(
 \frac{z_1}{t_1\delta_1},\ldots,
 \frac{z_n}{t_n\delta_n}
 \right)
 \right]\right).
 \end{aligned}
\end{equation}
Here the factor $V_\delta$ in $A^{\mathbf m}$ is cancelled by the
horizontal Jacobian $\prod_i(t_i\delta_i)$.

Since $B_\delta$ is supported in a fixed box, the horizontal support
of \eqref{eq:elementary-patch-kernel} is contained in
$|z_i|\le C\delta_i$. Moreover,
$t_d\Phi_\delta$ ranges in a fixed bounded interval $\mathcal I$.
Hence, by the Schwartz decay of $\check\psi$, for every $N\ge1$,
\[
 \sup_{t\in[1,2]^d}|K_{t,R}^{\mathbf m}(z)|
 \lesssim_N
 R\,1_{\{|z_i|\le C\delta_i,\ 1\le i\le n\}}
 \left(1+R\,\operatorname{dist}(z_d,\mathcal I)\right)^{-N}.
\]
For $N>1$, this envelope has $L^\infty$ norm $O(R)$ and
$L^1$ norm $O(RV_\delta)$. Therefore
\[
 \|M_{\loc}^{\mathbf m}P_R^{(d)}f\|_1
 \lesssim RV_\delta\|f\|_1,
 \qquad
 \|M_{\loc}^{\mathbf m}P_R^{(d)}f\|_\infty
 \lesssim R\|f\|_1.
\]

For each fixed $t$, the changes of variables
$z_i=t_i\delta_i v_i$ and
$u=R[z_d-t_d\Phi_\delta(v)]$ give
$\|K_{t,R}^{\mathbf m}\|_1\lesssim V_\delta$ uniformly in $t$.
Consequently,
\[
 \|M_{\loc}^{\mathbf m}P_R^{(d)}f\|_\infty
 \lesssim V_\delta\|f\|_\infty.
\]
This proves the remaining estimates in
Proposition~\ref{prop:dyadic-patch-bounds}.
\end{proof}

\subsection{Low-frequency estimates}
\label{sec:common-low-frequency}

We treat the low frequency in the distinguished coordinate of the fixed
graph patch, in every dimension $d\ge3$. Recall that
\[
 \psi_0(\eta)+\sum_{j\ge1}\psi(2^{-j}\eta)=1,
\]
and let $P_{\le1}^{(d)}$ have multiplier $\psi_0(\xi_d)$.
Its kernel is
\[
 K_{t,0}(z)
 =
 \left(\prod_{i=1}^nt_i^{-1}\right)
 \chi(z'/t')
 \check\psi_0\left(
 z_d-t_d\Phi(z'/t')
 \right).
\]
Since $\chi$ has compact support and $\Phi$ is bounded on this patch,
for every $N$,
\[
 \sup_{t\in[1,2]^d}|K_{t,0}(z)|
 \lesssim_N
 {\bf1}_{\{|z'|\le C\}}(1+|z_d|)^{-N}.
\]
Thus this envelope belongs to $L^s(\mathbb R^d)$ for every
$1\le s\le\infty$. Choosing $s$ so that
$1+1/q=1/p+1/s$, Young's inequality gives
\[
 \|M_{\loc}P_{\le1}^{(d)}f\|_q
 \lesssim_{p,q}\|f\|_p.
\]

\subsection{Interpolation and summation in dimensions \texorpdfstring{$d\ge4$}{d >= 4}}

For $d\ge4$, we explain how the corresponding six
estimates in Proposition~\ref{prop:dyadic-patch-bounds} lead to
Theorem~\ref{thm:lplq}. Put
\[
 u=\frac1p,\qquad v=\frac1q.
\]
For an estimate of the form
\[
 \|M_{\loc}^{\mathbf m}P_R^{(d)}f\|_q
 \lesssim
 R^aV_\delta^{1-u+v}
 \prod_{i=1}^n\lambda_i^{-c}\|f\|_p,
\]
we record the four exponents $(u,v,a,c)$. The six estimates valid
for $d\ge4$ in
Proposition~\ref{prop:dyadic-patch-bounds} correspond to
\[
 \begin{aligned}
 & \  \   \  \   (1,1,1,0),\qquad \quad  \ \  (1,0,1,0),\qquad  \qquad \  (0,0,0,0),\\
 &\left(\frac12,\frac12,\frac12,\frac14\right),\qquad
 \left(\frac12,\frac14,\frac12,\frac14\right),\qquad
 \left(\frac12,\frac1{2d},\frac{d-1}{d},\frac1d\right).
 \end{aligned}
\]
Hence interpolation gives the preceding estimate whenever
$(u,v,a,c)$ belongs to the convex hull of these six points.

For $1-u+v>0$ and $c\ge0$, each coordinate satisfies
\[
 \sum_{m\ge0}
 2^{-m(1-u+v)}
 \max\{1,R2^{-mr}\}^{-c}
 \lesssim
 (1+\log R)
 R^{-\min\{c,(1-u+v)/r\}}.
\]
Therefore the dyadic patches and the high-frequency pieces are
summable provided that, for some admissible $(a,c)$,
\[
 a<(d-1)c,
 \qquad
 ra<(d-1)(1-u+v).
\]
To eliminate the auxiliary exponents, let $\theta_1,\ldots,\theta_6$
be the convex weights of the six
estimates in the displayed order, and put
$h=\theta_4+\theta_5$ and $z=\theta_6$. Then
\[
 a=u+\frac{d-2}{2d}z,\qquad c=\frac h4+\frac zd.
\]
For fixed $u,v,h,z$, nonnegative weights with sum one exist
precisely when
\[
 \begin{gathered}
 h,z\ge0,\qquad h+z\le2u,\qquad h+z\le2(1-u),\\
 \frac h4+\frac z{2d}\le v
 \le u-\frac{d-1}{2d}z.
 \end{gathered}
\]
Indeed, after $h,z$ are fixed, $\theta_1$ ranges from zero to
$u-(h+z)/2$, $\theta_4$ ranges from zero to $h$, and
$v=\theta_1+\theta_4/4+h/4+z/(2d)$.
The two strict summability conditions become
\[
 (d-1)h+2z>4u,\qquad
 ru+\frac{r(d-2)}{2d}z<(d-1)(1-u+v).
\]
Write $n=d-1$ and $L=2d(u-nv)$. Eliminating $h$ gives
\[
 \begin{gathered}
 0\le z,\qquad L<z,\qquad
 z\le\frac{2d(u-v)}n,\\
 z<\min\left\{
 2u,\ 2(1-u),\ 2dv,\
 \frac{2n-2(d+1)u}{d-3},\
 \frac{2d[n-(r+n)u+nv]}{r(d-2)}
 \right\}.
 \end{gathered}
\]
The positivity of the last two upper bounds gives
$u<n/(d+1)$ and $v>((r+n)u-n)/n$. Comparing $L$ with the
upper bounds gives the three restrictions
\[
 v>\frac ud,\qquad
 v>\frac{nu-1}{d(d-3)},\qquad
 v>\frac{(r+1)u-1}{r(d-2)+1},
\]
and the additional inequality
$v>((d+1)u-1)/(dn)$.
This last restriction is redundant: its right-hand side is at
most $u/d$ when $u\le1/2$, and at most
$(nu-1)/[d(d-3)]$ when $u\ge1/2$.
Conversely, these strict restrictions and $0\le v\le u$ make
all the strict upper bounds positive. If $L<0$, take $z=0$;
if $L\ge0$, take $z$ slightly larger than $L$ and below all
the upper bounds. The corresponding interval for $h$ is then
nonempty. Thus the elimination gives
\begin{equation}
 \begin{gathered}
 0\le v\le u<\frac{d-1}{d+1},\\
 v>
 \max\left\{
 \frac ud,\,
 \frac{(d-1)u-1}{d(d-3)},\,
 \frac{(r+1)u-1}{r(d-2)+1},\,
 \frac{(r+d-1)u-(d-1)}{d-1}
 \right\}.
 \end{gathered}
\end{equation}
Equivalently,
\begin{equation}
 \begin{gathered}
 p>\frac{d+1}{d-1},
 \qquad
 p\le q<dp,\\
 \frac{d-1}{p}-\frac{d(d-3)}{q}<1,\\
 \frac{r+1}{p}-\frac{r(d-2)+1}{q}<1,\\
 \frac{r+d-1}{p}-\frac{d-1}{q}<d-1.
 \end{gathered}
\end{equation}

Combining these high-frequency estimates with the low-frequency bound
in Section~\ref{sec:common-low-frequency} proves the local estimate
on the fixed graph patch. Summing over the finite family of graph
patches gives Theorem~\ref{thm:lplq}.

\subsection{Interpolation and summation in dimension three}

In dimension three, \eqref{eq:dyadic-22} is critical on annuli
bounded away from the coordinate axes. We supplement it with the following
estimate. Its proof, including the quantitative dependence on the annular
scales, is given in Appendix~\ref{app:d3-smoothing}.

\begin{proposition}\label{prop:d3-annular-smoothing}
Let $d=3$ and $2\le r<\infty$. For every $\varepsilon>0$ there is a
finite $A=A(r,\varepsilon)$ such that, with
$\delta_* =\min(\delta_1,\delta_2)$,
\begin{equation}\label{eq:d3-annular-L6}
 \|M_{\loc}^{\mathbf m}P_R^{(3)}f\|_6
 \lesssim_{r,\varepsilon}
 \delta_*^{-A}R^{-1/6+\varepsilon}\|f\|_6,
 \qquad R\ge2.
\end{equation}
The constant is uniform in $R$ and $\mathbf m$ for the fixed graph
patch. Here $P_R^{(3)}$ localizes only the third frequency coordinate.
\end{proposition}

Only the polynomial loss in $\delta_*^{-1}$ is needed below. In particular,
we do not sum \eqref{eq:d3-annular-L6} over all annuli directly.

\begin{proposition}\label{prop:d3-decay}
Suppose $d=3$, $2\le r<\infty$, and
\begin{equation}\label{eq:d3-region}
 2<p<\infty,\qquad p\le q\le2p,\qquad
 b:=1-\frac{r+2}{2p}+\frac1q>0.
\end{equation}
For the local maximal operator on each fixed graph patch, there is
$\epsilon=\epsilon(p,q,r)>0$ such that
\begin{equation}\label{eq:d3-patch-decay}
 \|M_{\loc}P_R^{(3)}f\|_q
 \lesssim_{p,q,r} R^{-\epsilon}\|f\|_p,
 \qquad R\ge2.
\end{equation}
In particular, the diagonal estimate holds for
$p>\max\{2,r/2\}$.
\end{proposition}

\begin{proof}
Write $T_{\delta,R}=M_{\loc}^{\mathbf m}P_R^{(3)}$ and put
$u=1/p$, $v=1/q$. First interpolate \eqref{eq:dyadic-22},
\eqref{eq:dyadic-24}, and the $L^\infty$ bound in
Proposition~\ref{prop:dyadic-patch-bounds}. For
$0\le u\le1/2$ and $u/2\le v\le u$, this gives
\begin{align}
 \|T_{\delta,R}f\|_q
 &\lesssim
 R^u V_\delta^{1-u+v}
 (\lambda_1\lambda_2)^{-u/2}\|f\|_p
 \nonumber\\
 &\le V_\delta^{\,1-(r+2)u/2+v}\|f\|_p
 =V_\delta^b\|f\|_p.
 \label{eq:d3-all-scale}
\end{align}
Indeed, $\lambda_i\ge R\delta_i^r$, so the second inequality holds
whether or not a coordinate is active.

We next obtain frequency decay when both coordinates are active.
It suffices to keep track of $(u,v,a)$ in an estimate
\[
 \|T_{\delta,R}f\|_q
 \lesssim \delta_*^{-B}R^a\|f\|_p,
\]
where $B\ge0$ is finite and may depend on the fixed exponents. In the
region $R\delta_i^r\ge1$ for $i=1,2$, the estimates
\eqref{eq:dyadic-24}, \eqref{eq:dyadic-22},
\eqref{eq:dyadic-2-12}, \eqref{eq:d3-annular-L6}, and the
$L^\infty$ bound give, respectively,
\[
 \begin{aligned}
 \mathsf A&=(1/2,1/4,0),&
 \mathsf D&=(1/2,1/2,0),&
 \mathsf B&=(1/2,1/12,1/4),\\
 \mathsf C&=(1/6,1/6,-7/48),&
 \mathsf O&=(0,0,0).
 \end{aligned}
\]
Here we used $\varepsilon=1/48$ in
\eqref{eq:d3-annular-L6}. The remaining powers of $V_\delta$
in the first three estimates are bounded by a fixed power of
$\delta_*^{-1}$.

To include the line $q=2p$, note that
\[
 \mathsf H:=\tfrac13\mathsf B+\tfrac23\mathsf C
   =(5/18,5/36,-1/72).
\]
The segments from $\mathsf H$ to $\mathsf A$ and to $\mathsf O$
therefore give a strictly negative frequency exponent at every point
$0<u<1/2$, $v=u/2$. Likewise, the segments from $\mathsf C$ to
$\mathsf D$ and to $\mathsf O$ give a strictly negative exponent on
$0<u<1/2$, $v=u$. Interpolating between these two estimates at the
same $u$ covers $u/2\le v\le u$. Thus, for each $(p,q)$ in
\eqref{eq:d3-region}, there are $e=e(p,q)>0$ and
$B=B(p,q,r)\ge0$ such that
\begin{equation}\label{eq:d3-active-decay}
 \|T_{\delta,R}f\|_q
 \lesssim \delta_*^{-B}R^{-e}\|f\|_p
 \quad\text{if }R\delta_i^r\ge1\ (i=1,2).
\end{equation}
All these interpolations can be performed on arbitrary measurable
linearizations of the maximal operator, with constants independent
of the selectors.

Choose
\[
 0<\eta<\min\{1/r,\ e/(2(B+1))\}.
\]
For annuli with $\delta_*\le R^{-\eta}$, the geometric series in
\eqref{eq:d3-all-scale} gives
\[
 \sum_{\delta_*\le R^{-\eta}}
 \|T_{\delta,R}f\|_q
 \lesssim_b R^{-\eta b}\|f\|_p.
\]
If $\delta_*>R^{-\eta}$, both coordinates are active, and
\eqref{eq:d3-active-decay} gives
\[
 \sum_{\delta_*>R^{-\eta}}
 \|T_{\delta,R}f\|_q
 \lesssim (1+\log R)^2R^{-e+\eta B}\|f\|_p
 \lesssim_{p,q,r} R^{-e/4}\|f\|_p.
\]
Minkowski's inequality and
$M_{\loc}P_R^{(3)}f\le\sum_{\mathbf m}T_{\delta,R}f$
now prove \eqref{eq:d3-patch-decay}, for instance with
$\epsilon=\min\{\eta b,e/4\}>0$.
\end{proof}

Summing \eqref{eq:d3-patch-decay} over dyadic $R$, and using the
low-frequency kernel envelope already established in this section,
proves the local estimate on the fixed patch. A finite sum over graph
patches proves Theorem~\ref{thm:lplq-d3}. The diagonal specialization,
together with the coordinate-frequency transfer in
Lemma~\ref{lem:global-decay}, supplies the local high-frequency input
for the global theorem when $d=3$.

\section{Normalized \texorpdfstring{$L^2$}{L2}-improving estimates}
\label{sec:proof-local-patch}

We prove Lemma~\ref{lem:normalized-improving}. The multiplier
factorization reduces the phase to one-dimensional power phases.
A fixed-height $TT^*$ estimate gives the $L^4$ bound; a common full
spatial kernel estimate gives the $L^{2d}$ bound for $d\ge4$ and
the $L^{12}$ bound for $d=3$. We then sum the sectors and recover
the companion $L^2$ bound by spatial localization.

By Fourier inversion,
\[
 A_\delta P_R^{(d)}f(x,t)
 =
 \int_{\mathbb R^d}
 e^{ix\cdot\xi}
 m_{t,\delta}(\xi)
 \psi(\xi_d/R)
 \widehat f(\xi)\,d\xi,
\]
where
\[
 m_{t,\delta}(\xi)
 =
 V_\delta\int_{\mathbb R^n}
 e^{-iF_{t,\xi}(v)}B_\delta(v)\,dv,
 \qquad
 F_{t,\xi}(v)
 =
 \sum_{i=1}^n t_i\xi_i v_i
 +t_d\xi_d\Phi_\delta(v).
\]
Fixed Fourier normalization constants are suppressed in the
oscillatory integral representations and absorbed into the amplitudes
and implicit constants.
After splitting the annular cutoff into the two signs of $\xi_d$, it
suffices to consider $\xi_d>0$. The case $\xi_d<0$ is treated in the
same way after reversing all frequency signs.

\subsection{The canonical phase of the \texorpdfstring{$\ell^r$}{ell\string^r} sphere}

Let
\[
 h(\xi)
 =
 \left(\sum_{j=1}^d|\xi_j|^{r'}\right)^{1/r'},
 \qquad
 r'=\frac{r}{r-1}.
\]
Consider the positive graph
\[
 y(v)=(v,\Phi(v)),\qquad
 \Phi(v)=\left(1-\sum_{i=1}^n v_i^r\right)^{1/r},
 \qquad v_i>0,\quad \Phi(v)>0.
\]
For this preliminary calculation, take
$\chi\in C_c^\infty(\Omega_+)$, where
\[
 \Omega_+=\left\{v:v_i>0,\ \sum_i v_i^r<1\right\}.
\]
Write $d\mu_\chi=\chi(v)\,dv$. Then
\begin{equation}
 \widehat{d\mu_\chi}(R\xi)
 =
 \int_{\mathbb R^n}
 e^{-iR(\xi'\cdot v+\xi_d\Phi(v))}
 \chi(v)\,dv.
\end{equation}

On the positive frequency orthant, the critical point equations are
\[
 \xi_i
 =
 \xi_d v_i^{r-1}\Phi^{1-r},
 \qquad 1\le i\le n.
\]
Hence
\[
 \frac{v_i}{\Phi}
 =
 \left(\frac{\xi_i}{\xi_d}\right)^{r'-1},
\]
and the unique critical point is
\begin{equation}
 v_i^*
 =
 \left(\frac{\xi_i}{h(\xi)}\right)^{r'-1},
 \qquad
 \Phi_*
 =
 \left(\frac{\xi_d}{h(\xi)}\right)^{r'-1}.
\end{equation}
Its critical value is
\begin{align}
 \xi'\cdot v^*+\xi_d\Phi_*
 &=
 \frac{\sum_{i=1}^n\xi_i^{r'}+\xi_d^{r'}}
 {h(\xi)^{r'-1}}
 =
 h(\xi).
\end{align}
Thus, on a fixed compact positive frequency set for which the critical
point remains in the patch, stationary phase gives
\begin{equation}
 \widehat{d\mu_\chi}(R\xi)
 =
 R^{-n/2}e^{-iRh(\xi)}a_R(\xi),
 \qquad
 \|a_R\|_{C^N}
 \lesssim_N
 \|\chi\|_{C^{M(N)}}.
\end{equation}

The same sign condition appears for the actual phase $F_{t,\xi}$.
Indeed,
\[
 \partial_{v_i}F_{t,\xi}(v)
 =
 t_i\xi_i
 -
 t_d\xi_d\delta_i^r
 v_i^{r-1}\Phi_\delta(v)^{1-r}.
\]
Since we are considering $\xi_d>0$, a stationary point in the $v_i$
variable can occur only when $\xi_i>0$. If $\xi_i<0$, then
\[
 |\partial_{v_i}F_{t,\xi}(v)|
 =
 t_i|\xi_i|
 +
 t_d\xi_d\delta_i^r
 v_i^{r-1}\Phi_\delta(v)^{1-r}
 \gtrsim
 |\xi_i|+R\delta_i^r,
\]
so the corresponding frequency sector is nonstationary.

We treat mixed-sign frequency sectors as nonstationary.
To track the degeneracy scales $\delta_i$, we now derive an exact
sector factorization for $m_{t,\delta}$.

\subsection{Exact factorization of \texorpdfstring{$m_{t,\delta}$}{m(t, delta)}}

We now fix the rescaled dyadic patch and decompose the horizontal
frequencies according to their stationary scales. Recall that $\lambda_i=\max\{1,R\delta_i^r\},
 I=\{i:R\delta_i^r\ge1\}.$
We decompose $\xi_i$ dyadically according to $\lambda_i$:
\begin{align}
 S_{i,0}
 &=
 \{\xi_i:|\xi_i|\le C_0\lambda_i\},\\
 S_{i,k}
 &=
 \{\xi_i:
 2^{k-1}C_0\lambda_i<|\xi_i|
 \le2^kC_0\lambda_i\},
 \qquad k\ge1,
\end{align}
where $C_0$ is a sufficiently large fixed constant.

For $i\in I$, we further split the bounded sector by setting
\begin{align}
 S_{i,0}^{\mathrm{stat}}
 &:=
 S_{i,0}\cap
 \{\xi_i:
 cR\delta_i^r<\xi_i<c^{-1}R\delta_i^r\},
 \\
 S_{i,0}^{\mathrm{non}}
 &:=
 S_{i,0}\setminus S_{i,0}^{\mathrm{stat}}.
\end{align}
Here $c>0$ is chosen sufficiently small and $C_0$ sufficiently large
so that $S_{i,0}^{\mathrm{stat}}$ contains every possible stationary
frequency in the $v_i$ variable. A critical point need not lie in the
amplitude support. For a fixed frequency sector, let
$S\subset I$ be the set of indices for which
$\xi_i\in S_{i,0}^{\mathrm{stat}}$.

The indices are therefore divided into four classes:
\begin{itemize}
\item $k_i\ge1$: remote and nonstationary;
\item $i\in I\setminus S$ and $k_i=0$: active and nonstationary;
\item $i\in S$: active and stationary;
\item $i\notin I$ and $k_i=0$: inactive and bounded.
\end{itemize}

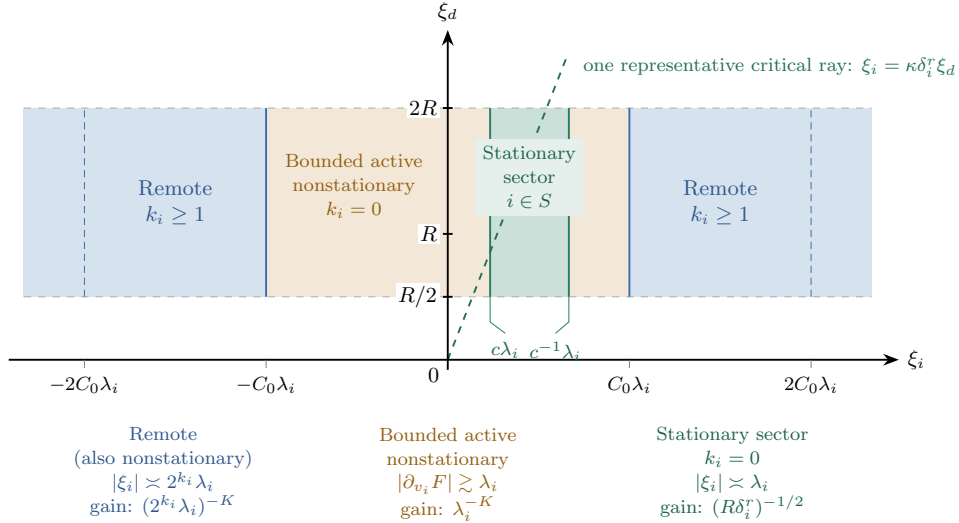
\begin{figure}[H]
\centering
\resizebox{\linewidth}{!}{%
\begin{tikzpicture}[x=0.89cm,y=1.85cm,>=Stealth,
                    every node/.style={font=\small}]

 \fill[sectornonstationary!17] (-7,.5) rectangle (7,2);
 \fill[sectorremote!20] (-7,.5) rectangle (-3,2);
 \fill[sectorremote!20] (3,.5) rectangle (7,2);
 \fill[sectorstationary!24] (.70,.5) rectangle (2.00,2);

 \draw[gray!65,dashed] (-7,.5)--(7,.5);
 \draw[gray!65,dashed] (-7,2)--(7,2);

 \draw[sectorremote,thick] (-3,.5)--(-3,2);
 \draw[sectorremote,thick] (3,.5)--(3,2);
 \draw[sectorremote,densely dashed] (-6,.5)--(-6,2);
 \draw[sectorremote,densely dashed] (6,.5)--(6,2);

 \draw[sectorstationary,thick] (.70,.5)--(.70,2);
 \draw[sectorstationary,thick] (2.00,.5)--(2.00,2);

 \draw[sectorstationary!90!black,thick,dashed] (0,0)--(1.96,2.4);
 \node[sectorstationary!90!black,anchor=west,font=\footnotesize]
       at (2.1,2.32) {one representative critical ray: $\xi_i=\kappa\delta_i^r\xi_d$};

 \draw[->,thick] (-7.25,0)--(7.45,0) node[right] {$\xi_i$};
 \draw[->,thick] (0,-.08)--(0,2.62) node[above] {$\xi_d$};
 \node[below left] at (0,0) {$0$};

 \foreach \yy/\lab in {.5/{R/2},1/{R},2/{2R}} {
   \draw[thick] (-.08,\yy)--(.08,\yy);
   \node[anchor=east,fill=white,inner sep=1pt] at (-.13,\yy) {$\lab$};
 }

 \foreach \xx/\lab in {-6/{-2C_0\lambda_i},-3/{-C_0\lambda_i},3/{C_0\lambda_i},6/{2C_0\lambda_i}} {
   \draw[gray] (\xx,0)--(\xx,-.065);
   \node[below,font=\footnotesize] at (\xx,-.07) {$\lab$};
 }

 \node[align=center,sectorremote!80!black] at (-4.5,1.25)
       {Remote\\$k_i\ge1$};

 \node[align=center,sectornonstationary!80!black,font=\footnotesize]
       at (-1.55,1.38) {Bounded active\\nonstationary\\$k_i=0$};

 \node[align=center,sectorstationary!80!black,font=\footnotesize,
       fill=sectorstationary!12,inner sep=3pt] at (1.35,1.45)
       {Stationary\\sector\\$i\in S$};

 \node[align=center,sectorremote!80!black] at (4.5,1.25)
       {Remote\\$k_i\ge1$};

 \draw[sectorstationary] (.70,.5)--(.70,.30)--(.95,.21);
 \node[below,font=\footnotesize,sectorstationary!80!black]
       at (.95,.21) {$c\lambda_i$};

 \draw[sectorstationary] (2.00,.5)--(2.00,.30)--(1.78,.21);
 \node[below,font=\footnotesize,sectorstationary!80!black]
       at (1.78,.21) {$c^{-1}\lambda_i$};

 \node[anchor=north,align=center,sectorremote!80!black,font=\footnotesize]
       at (-4.7,-.47) {Remote\\(also nonstationary)\\
       $|\xi_i|\asymp 2^{k_i}\lambda_i$\\gain: $(2^{k_i}\lambda_i)^{-K}$};

 \node[anchor=north,align=center,sectornonstationary!80!black,font=\footnotesize]
       at (0,-.47) {Bounded active\\nonstationary\\
       $|\partial_{v_i}F|\gtrsim \lambda_i$\\gain: $\lambda_i^{-K}$};

 \node[anchor=north,align=center,sectorstationary!80!black,font=\footnotesize]
       at (4.7,-.47) {Stationary sector\\$k_i=0$\\
       $|\xi_i|\asymp \lambda_i$\\gain: $(R\delta_i^r)^{-1/2}$};
\end{tikzpicture}%
}
\caption{Frequency decomposition for the phase
$\sum_{i=1}^n t_i\xi_i v_i+t_d\xi_d\Phi_\delta(v)$ on $\xi_d>0$.}
\end{figure}

The displayed sector sets specify the frequency scales. In the actual
partition we use overlapping smooth cutoffs, as follows. Let
$\chi_b\in C_c^\infty(\mathbb R)$ equal one on $[-1,1]$ and be
supported in $[-2,2]$, and set
\[
 \beta_{i,0}(\xi_i)=\chi_b\!\left(\frac{\xi_i}{C_0\lambda_i}\right),
 \qquad
 \beta_{i,k}(\xi_i)=
 \chi_b\!\left(\frac{\xi_i}{2^kC_0\lambda_i}\right)
 -\chi_b\!\left(\frac{\xi_i}{2^{k-1}C_0\lambda_i}\right),
 \quad k\ge1.
\]
Their sum is one. Choose $c>0$ small enough that every possible
stationary ratio $\xi_i/(R\delta_i^r)$ lies in the interior of
$[c,c^{-1}]$, with a fixed margin. This choice also covers the phases
obtained by deleting previously eliminated nonstationary coordinates:
deleting such terms increases the graph height, which remains in a
fixed positive interval. For each active coordinate, choose a fixed
smooth function $\theta$ equal one on $[c,c^{-1}]$ and supported in
$[c/2,2c^{-1}]$, and split the bounded piece into
\[
 \beta_{i,0}\theta\!\left(\frac{\xi_i}{R\delta_i^r}\right)
 +\beta_{i,0}\left[1-\theta\!\left(\frac{\xi_i}{R\delta_i^r}\right)\right].
\]
Take $C_0$ sufficiently large relative to this fixed interval. The
second piece is uniformly nonstationary on the amplitude support,
and the pieces with $k_i\ge1$ are uniformly remote. After
$\xi_i=2^{k_i}\lambda_i\eta_i$, all cutoff derivatives are uniformly
bounded. We use the same sector notation for these smooth pieces.

The following factorization records the decay from the first three
classes; the bounded inactive variables contribute no additional
oscillatory decay.

\begin{lemma}
\label{lem:local-exact-sector}
Fix $t\in[1,2]^d$. On the support of $\psi(\xi_d/R)$ with
$\xi_d>0$, and for any fixed integers $K,N\ge1$, $m_{t,\delta}$
is a sum over frequency sectors $(S,k)$, where $S\subset I$ and
$k\in\mathbb Z_{\ge0}^n$, of terms of the form
\begin{equation}\label{eq:sector-factorization}
 V_\delta
 \prod_{i\in S}(R\delta_i^r)^{-1/2}
 \prod_{\substack{i\in I\setminus S\\ k_i=0}}
 (R\delta_i^r)^{-K}
 \prod_{i:k_i\ge1}(2^{k_i}\lambda_i)^{-K}
 e^{-ih_S(t,\eta)}
 a_{R,\delta,S,k}(t,\eta),
\end{equation}
where
\begin{align}\label{eq:local-dual-phase}
\begin{split}
 h_S(t,\eta)
 &=
 \left(
 |t_d\xi_d|^{r'}
 +
 \sum_{i\in S}|t_i\xi_i/\delta_i|^{r'}
 \right)^{1/r'}\\
 &=
 R\left\{
 t_d^{r'}\left[
 |\eta_d|^{r'}
 +
 \sum_{i\in S}
 \delta_i^r(t_i/t_d)^{r'}|\eta_i|^{r'}
 \right]
 \right\}^{1/r'}.
\end{split}
\end{align}
Here
\[
 \xi_d=R\eta_d,\qquad
 \xi_i=\omega_i\eta_i,\qquad
 \omega_i=2^{k_i}\lambda_i,
\]
and $k_i=0$ for every $i\in S$. The amplitudes are supported in
one fixed compact box in $\eta$, independent of the scales. Moreover,
\begin{equation}\label{eq:local-sector-amplitude}
 \|a_{R,\delta,S,k}\|_{C_{(t,\eta)}^N}
 \le
 C_{K,N}\|B_\delta\|_{C^{M(K,N,d)}}.
\end{equation}
The constants and the finite integer $M(K,N,d)$ are independent
of $R$, $\delta$, $S$, $k$, and the relative sizes of the $\lambda_i$.
They may depend on the fixed $d,r$, support margins, and cutoffs.
\end{lemma}

We postpone the proof of Lemma~\ref{lem:local-exact-sector} to
Section~\ref{sec:proof-lem-exact}. With $\omega_i$ as above, let
\[
 \mathfrak D=\diag(\omega_1,\ldots,\omega_n,R)
\]
and make the change of variables $\xi=\mathfrak D\eta$. Each frequency
sector is then reduced to
\begin{align}\label{reduced-oio}
 &V_\delta
 \prod_{i\in S}(R\delta_i^r)^{-1/2}
 \prod_{\substack{i\in I\setminus S\\ k_i=0}}
 (R\delta_i^r)^{-K}
 \prod_{i:k_i\ge1}(2^{k_i}\lambda_i)^{-K}
 \left(R\prod_{i=1}^n\omega_i\right) \nonumber\\
 &\qquad\times
 \int_{\mathbb R^d}
 e^{i(\mathfrak D x\cdot\eta-h_S(t,\eta))}
 a_{R,\delta,S,k}(t,\eta)
 \psi(\eta_d)
 \widehat f(\mathfrak D\eta)\dd\eta.
\end{align}

We next express $e^{-ih_S(t,\eta)}$ by scalar Fourier inversion.
The resulting nonlinear phase separates into one-dimensional power
phases.

\subsection{Scalar Fourier inversion}

In the previous subsection, each frequency sector was reduced to a multiplier
\[
 C_{R,\delta,S,k}e^{-ih_S(t,\eta)}a_{R,\delta,S,k}(t,\eta),
\]
where $h_S(t,\eta)=R\bigl(t_d^{r'}|\eta_d|^{r'}+\sum_{i\in S}\delta_i^r
t_i^{r'}|\eta_i|^{r'}\bigr)^{1/r'}$ and
\[
 C_{R,\delta,S,k}
 =
 V_\delta
 \prod_{i\in S}(R\delta_i^r)^{-1/2}
 \prod_{\substack{i\in I\setminus S\\ k_i=0}}(R\delta_i^r)^{-K}
 \prod_{i:k_i\ge1}(2^{k_i}\lambda_i)^{-K}.
\]
We remove the outer $1/r'$-power by scalar Fourier inversion.

For $i\in S$, one has $k_i=0$, $\lambda_i=R\delta_i^r$, and
$\xi_i=R\delta_i^r\eta_i$. Since also $|\xi_d|\sim R$, we have
$|\eta_d|,|\eta_i|\sim1$ for $i\in S$. Hence
$Z:=t_d^{r'}\bigl(|\eta_d|^{r'}+\sum_{i\in S}
\delta_i^r(t_i/t_d)^{r'}|\eta_i|^{r'}\bigr)$ ranges in a fixed compact
interval $\mathcal I\Subset(0,\infty)$. Choose
$\beta\in C_c^\infty((0,\infty))$ with $\beta=1$ on $\mathcal I$ and set
$k_R(s):=(2\pi)^{-1}\int e^{-isw}\beta(w)e^{-iRw^{1/r'}}\,dw$.

\begin{lemma}\label{lem:scalar-kernel}
Uniformly in $R\ge2$,
\begin{equation}\label{eq:scalar-inversion}
 e^{-ih_S(t,\eta)}
 =
 \int_{\R}t_d^{-r'}k_R\left(\frac{s}{t_d^{r'}}\right)
 e^{is\left(
 |\eta_d|^{r'}+
 \sum_{i\in S}\delta_i^r(t_i/t_d)^{r'}|\eta_i|^{r'}
 \right)}\,ds.
\end{equation}
Moreover, $\|k_R\|_2\lesssim1$ and, for every $N\ge1$,
$|k_R(R\sigma)|\lesssim_N R^{-1/2}(1+|\sigma|)^{-N}$. Equivalently,
\begin{equation}\label{eq:scalar-inversion-rescaled}
 e^{-ih_S(t,\eta)}
 =
 R\int_{\R}k_R(R\sigma)e^{iR\sigma Q_t(\eta)}\,d\sigma,
 \qquad
 Q_t(\eta)
 =
 t_d^{r'}|\eta_d|^{r'}
 +
 \sum_{i\in S}\delta_i^rt_i^{r'}|\eta_i|^{r'}.
\end{equation}
\end{lemma}

\begin{proof}
Since $h_S(t,\eta)=RZ^{1/r'}$ and $\beta(Z)=1$, Fourier inversion gives
\[
 e^{-ih_S(t,\eta)}=\int_\R k_R(s)e^{isZ}\,ds.
\]
The change of variables $s\mapsto s/t_d^{r'}$ gives
\eqref{eq:scalar-inversion}, while Plancherel gives
\[
 \|k_R\|_2\lesssim\|\beta(w)e^{-iRw^{1/r'}}\|_2=\|\beta\|_2.
\]

For the pointwise bound, put $\phi_\sigma(w)=\sigma w+w^{1/r'}$.
Then
\[
 k_R(R\sigma)=(2\pi)^{-1}
 \int e^{-iR\phi_\sigma(w)}\beta(w)\,dw.
\]
On $\supp\beta$,
$|\phi_\sigma''(w)|\asymp1$, so van der Corput gives
$|k_R(R\sigma)|\lesssim R^{-1/2}$ for bounded $\sigma$. If $|\sigma|$
is sufficiently large, then $|\phi_\sigma'(w)|\gtrsim|\sigma|$ on
$\supp\beta$, and repeated integration by parts gives
$|k_R(R\sigma)|\lesssim_N R^{-N}|\sigma|^{-N}$. Combining the two
estimates yields the stated bound. Finally, $s=Rt_d^{r'}\sigma$ in
\eqref{eq:scalar-inversion} gives \eqref{eq:scalar-inversion-rescaled}.
\end{proof}

By \eqref{eq:scalar-inversion}, \eqref{reduced-oio} becomes
\begin{align}\label{reduced-oio2}
 &C_{R,\delta,S,k}
 \left(R\prod_{i=1}^n\omega_i\right)
 \int_{\mathbb R}
 t_d^{-r'}k_R\left(\frac{s}{t_d^{r'}}\right)
 \int_{\mathbb R^d}
 e^{i\left(
 \mathfrak D x\cdot\eta+
 s|\eta_d|^{r'}+
 s\sum_{i\in S}\delta_i^r(t_i/t_d)^{r'}|\eta_i|^{r'}
 \right)}
 \nonumber\\
 &\hspace{35mm}\times
 a_{R,\delta,S,k}(t,\eta)\psi(\eta_d)
 \widehat f(\mathfrak D\eta)\,d\eta ds.
\end{align}
For the $L^{2d}$ estimate we shall instead use the equivalent form
\begin{align}\label{reduced-oio-2d}
 &C_{R,\delta,S,k}
 \left(R\prod_{i=1}^n\omega_i\right)R
 \int_{\R}k_R(R\sigma)
 \int_{\R^d}
 e^{i(\mathfrak D x\cdot\eta+R\sigma Q_t(\eta))}
 a_{R,\delta,S,k}(t,\eta)\psi(\eta_d)
 \widehat f(\mathfrak D\eta)\,d\eta d\sigma .
\end{align}

\subsection{The fixed-height \texorpdfstring{$TT^*$}{TT*} estimate}

We use a $TT^*$ argument to obtain the estimate needed for
\eqref{eq:normalized-L4-bound}. We first record a simple oscillatory
integral estimate.

\begin{lemma}\label{lem:mixed-linear-power}
Let $r'=r/(r-1)\in(1,2]$, and let $\vartheta$ be a smooth cutoff
supported in a compact subinterval of either $(0,\infty)$ or
$(-\infty,0)$. Then
\begin{equation}\label{eq:mixed-linear-power}
 \left|
 \int e^{i(Av+B|v|^{r'})}\vartheta(v)\dd v
 \right|
 \le
 C(1+|A|)^{-1/2},
 \qquad A,B\in\R.
\end{equation}
The estimate is uniform for families of amplitudes with uniformly
bounded $C^1$ norms and support in a common fixed interval.
\end{lemma}

\begin{proof}
For $|A|\le1$, the estimate is trivial. Suppose $|A|>1$.
If $|B|\le c|A|$, with $c>0$ sufficiently small depending only on the
support of $\vartheta$, then
\[
 \left|
 A+Br'\sgn(v)|v|^{r'-1}
 \right|
 \ge \frac{|A|}{2}
\]
on the support. Integration by parts gives $C|A|^{-1}$.
If $|B|>c|A|$, then
\[
 \left|
 \frac{d^2}{dv^2}(Av+B|v|^{r'})
 \right|
 =
 |B|r'(r'-1)|v|^{r'-2}
 \gtrsim |A|,
\]
and the second derivative estimate gives $C|A|^{-1/2}$.
This proves \eqref{eq:mixed-linear-power}.
\end{proof}

We now reduce \eqref{reduced-oio2} to a fixed-height estimate. Fix a
sector $(S,k)$, put $J_R=R\prod_{i=1}^n\omega_i$, and set
$g(\eta)=\widehat f(\mathfrak D\eta)$. After linearizing the parameter
supremum by an arbitrary measurable function $t=t(x)\in[1,2]^d$, fix
$x_d=c$ and write $t(x',c)=t$. The contribution of the sector
$(S,k)$ in \eqref{reduced-oio2} is then
\begin{equation}
 C_{R,\delta,S,k}J_R
 \int_{\R}
 t_d(x',c)^{-r'}
 k_R\left(\frac{s}{t_d(x',c)^{r'}}\right)
 \mathcal T_cg(x',s)\,ds,
\end{equation}
where
\begin{align}
 \mathcal T_cg(x',s)
 =
 \int_{\R^d}
 &e^{i\left[
 \sum_{i=1}^n\omega_i x_i\eta_i
 +Rc\eta_d
 +s|\eta_d|^{r'}
 +s\sum_{i\in S}
 \delta_i^r
 \left(\frac{t_i(x',c)}{t_d(x',c)}\right)^{r'}
 |\eta_i|^{r'}
 \right]}
 \nonumber\\
 &\times
 a_{R,\delta,S,k}(t(x',c),\eta)
 \psi(\eta_d)g(\eta)\,d\eta.
\end{align}
Since $t_d\in[1,2]$,
$\|t_d^{-r'}k_R(\,\cdot\,/t_d^{r'})\|_2
=t_d^{-r'/2}\|k_R\|_2\lesssim1$. Hence Cauchy--Schwarz in $s$ gives
\begin{equation}
 \left|
 \int_{\R}
 t_d^{-r'}k_R\left(\frac{s}{t_d^{r'}}\right)
 \mathcal T_cg(x',s)\,ds
 \right|
 \lesssim
 \|\mathcal T_cg(x',\cdot)\|_{L_s^2}.
\end{equation}
It therefore remains to estimate $\mathcal T_c$ uniformly in $c$ and
in the measurable selector $t(x',c)$. No spatial derivatives of this
selector will be used.

\begin{lemma}\label{lem:fixed-height-L4}
For some finite $M=M(d,r)$,
\begin{equation}\label{eq:fixed-height-T}
 \|\mathcal T_cg\|_{L^4_{x'}L^2_s}
 \lesssim
 \mathfrak B_M
 \prod_{i=1}^n\omega_i^{-1/4}
 \|g\|_2,
\end{equation}
uniformly in $c$, $R$, $\delta$, $S$, $k$, and the measurable function
$t=t(x',c)$.
\end{lemma}

\begin{proof}
After a finite smooth refinement of the frequency sector, we may assume
that $|\eta_d|\asymp1$ and, for each $i\in S$, $|\eta_i|\asymp1$ with
a fixed sign on the support. The resulting amplitudes have uniformly
bounded $C^N$ norm controlled by $\mathfrak B_M$.

The kernel of $\mathcal T_c\mathcal T_c^*$ is
\begin{align}
 K_c(x',s;z',u)
 =
 \int_{\R^d}
 &e^{i\left[
 \sum_{i=1}^n\omega_i(x_i-z_i)\eta_i
 +(s-u)|\eta_d|^{r'}
 \right]}
 \nonumber\\
 &\times
 e^{i\sum_{i\in S}\delta_i^r
 \left[
 s\left(\frac{t_i(x',c)}{t_d(x',c)}\right)^{r'}
 -
 u\left(\frac{t_i(z',c)}{t_d(z',c)}\right)^{r'}
 \right]|\eta_i|^{r'}}
 \nonumber\\
 &\times
 a_{R,\delta,S,k}(t(x',c),\eta)
 \overline{a_{R,\delta,S,k}(t(z',c),\eta)}
 |\psi(\eta_d)|^2\dd\eta.
\end{align}
The factors $e^{iRc\eta_d}$ cancel exactly.

Since the phase is separated in the variables $\eta_i$, the
one-dimensional estimates may be applied successively. Indeed, after
integrating in one variable, derivatives in the remaining variables
pass through that integral, so the resulting amplitude satisfies the
same type of bounds, with higher mixed derivatives of the original
amplitude. Since the mixed $C^N$ norms of
$a_{R,\delta,S,k}$ are uniformly bounded, iterating the
one-dimensional estimates gives
\begin{align}
 |K_c(x',s;z',u)|
 \lesssim_N{}&
 \mathfrak B_M^2
 (1+|s-u|)^{-N}
 \prod_{i\in S}
 (1+\omega_i|x_i-z_i|)^{-1/2}
 \nonumber\\
 &\times
 \prod_{i\notin S}
 (1+\omega_i|x_i-z_i|)^{-N}.
\end{align}
Consequently,
\begin{equation}
 |K_c(x',s;z',u)|
 \lesssim
 \mathfrak B_M^2
 (1+|s-u|)^{-N}
 \prod_{i=1}^n
 \omega_i^{-1/2}|x_i-z_i|^{-1/2}.
\end{equation}

For $H\in L^{4/3}_{z'}L^2_u$, it follows that
\begin{align*}
 |\mathcal T_c\mathcal T_c^*H(x',s)|
 &\lesssim
 \mathfrak B_M^2
 \prod_{i=1}^n\omega_i^{-1/2}\\
 &\quad\times
 \int_{\R^n}
 \prod_{i=1}^n|x_i-z_i|^{-1/2}
 \int_{\R}
 (1+|s-u|)^{-N}|H(z',u)|\dd u\dd z'.
\end{align*}
Young's inequality in the $s$-variable gives
\begin{align}
 \|\mathcal T_c\mathcal T_c^*H(x',\cdot)\|_{L^2_s}
 \lesssim
 \mathfrak B_M^2
 \prod_{i=1}^n\omega_i^{-1/2}
 \int_{\R^n}
 \prod_{i=1}^n|x_i-z_i|^{-1/2}
 \|H(z',\cdot)\|_{L^2_u}\dd z'.
\end{align}
The one-dimensional Hardy--Littlewood--Sobolev inequality gives
\[
 \left\|
 \int_{\R}
 |x_i-z_i|^{-1/2}F(z_i)\dd z_i
 \right\|_{L^4_{x_i}}
 \lesssim
 \|F\|_{L^{4/3}_{z_i}}.
\]
Iterating this estimate in $i=1,\ldots,n$, using Minkowski's inequality
between successive applications (with $4/3\le4$), yields
\begin{equation}
 \left\|
 \int_{\R^n}
 \prod_{i=1}^n|x_i-z_i|^{-1/2}F(z')\dd z'
 \right\|_{L^4_{x'}}
 \lesssim
 \|F\|_{L^{4/3}_{z'}}.
\end{equation}
Therefore
\[
 \|\mathcal T_c\mathcal T_c^*H\|_{L^4_{x'}L^2_s}
 \lesssim
 \mathfrak B_M^2
 \prod_{i=1}^n\omega_i^{-1/2}
 \|H\|_{L^{4/3}_{z'}L^2_u}.
\]
By duality,
\[
 \|\mathcal T_cg\|_{L^4_{x'}L^2_s}
 \lesssim
 \mathfrak B_M
 \prod_{i=1}^n\omega_i^{-1/4}
 \|g\|_2,
\]
which proves \eqref{eq:fixed-height-T}.
\end{proof}

\subsection{Full spatial \texorpdfstring{$TT^*$}{TT*} estimates}

For the $L^{2d}$ estimate when $d\ge4$, and the $L^{12}$ estimate
when $d=3$, we retain all spatial variables.
We linearize the parameter supremum by an arbitrary measurable function
$t=t(x)\in[1,2]^d$ and write
\begin{align}
 \mathcal R_{R,\delta,S,k}g(x)
 =
 R\int_{\R}k_R(R\sigma)
 \int_{\R^d}
 e^{i(\mathfrak D x\cdot\eta+R\sigma Q_{t(x)}(\eta))}
 a_{R,\delta,S,k}(t(x),\eta)\psi(\eta_d)g(\eta)
 \,d\eta d\sigma,
\end{align}
where
\[
 Q_t(\eta)=t_d^{r'}|\eta_d|^{r'}+
 \sum_{i\in S}\delta_i^rt_i^{r'}|\eta_i|^{r'}.
\]
Thus \eqref{reduced-oio-2d} is
\[
 C_{R,\delta,S,k}\left(R\prod_{i=1}^n\omega_i\right)
 \mathcal R_{R,\delta,S,k}g,
\]
with $g(\eta)=\widehat f(\mathfrak D\eta)$.

\begin{lemma}\label{lem:sector-kernel}
Let $d\ge3$ and let $K(x,y)$ be the kernel of
$\mathcal R_{R,\delta,S,k}\mathcal R_{R,\delta,S,k}^*$.
Put $h=x-y$, $\omega_d=R$, and $\mathcal A=S\cup\{d\}$.
For some finite $M=M(d,r)$, uniformly in $R$, $\delta$, $S$, $k$,
and the measurable selector $t=t(x)$,
\begin{align}\label{eq:L2d-kernel-before-HLS}
 |K(x,y)|
 \lesssim
 \mathfrak B_M^2 R
 \prod_{j\in\mathcal A}
 \omega_j^{-1/d}
 (1+\omega_j|h_j|)^{-(d-2)/(2d)}
 \prod_{\substack{j\le n\\j\notin S}}
 (1+\omega_j|h_j|)^{-2}.
\end{align}
\end{lemma}

\begin{proof}
Put $\omega_d=R$, $h=x-y$, and $\mathcal A=S\cup\{d\}$. Since
$\omega_i=R\delta_i^r=\lambda_i$ for $i\in S$, the kernel of
$\mathcal R_{R,\delta,S,k}\mathcal R_{R,\delta,S,k}^*$ is
\begin{align}
 K(x,y)
 =
 R^2\iint_{\R^2}
 &k_R(R\sigma)\overline{k_R(R\rho)}
 \int_{\R^d}
 e^{i\sum_{j=1}^d\omega_jh_j\eta_j}
 \nonumber\\
 &\times
 e^{i\sum_{j\in\mathcal A}
 \omega_j[\sigma t_j(x)^{r'}-\rho t_j(y)^{r'}]|\eta_j|^{r'}}
 A_{x,y}(\eta)\,d\eta\,d\sigma d\rho ,
\end{align}
where
$A_{x,y}(\eta)=a_{R,\delta,S,k}(t(x),\eta)
\overline{a_{R,\delta,S,k}(t(y),\eta)}|\psi(\eta_d)|^2$.

After a finite smooth refinement, the variables $\eta_j$, $j\in\mathcal A$,
have fixed sign and stay in fixed compact intervals away from zero. Choose
fixed one-dimensional cutoffs $\vartheta_j$ which equal one on the
corresponding frequency supports. Expanding $A_{x,y}$ in a Fourier series
on a fixed box, we may write
\[
 A_{x,y}(\eta)
 =
 \sum_{\ell\in\mathbb Z^d}c_\ell(x,y)
 \prod_{j=1}^d\vartheta_j(\eta_j)e^{i\pi\ell_j\eta_j/L},
 \qquad
 \sum_{\ell}|c_\ell(x,y)|
 \prod_{j=1}^d(1+|\ell_j|)^2
 \lesssim\mathfrak B_M^2.
\]
For $j\in\mathcal A$, set
\[
 J_{j,\ell_j}(\sigma,\rho)
 =
 \int
 e^{i\omega_j\{h_j\eta+
 [\sigma t_j(x)^{r'}-\rho t_j(y)^{r'}]|\eta|^{r'}\}}
 \vartheta_j(\eta)e^{i\pi\ell_j\eta/L}\,d\eta .
\]
Lemma~\ref{lem:mixed-linear-power}, applied with the last exponential
included in the amplitude, gives
\begin{equation}\label{eq:L2d-J-sup}
 \sup_\sigma|J_{j,\ell_j}(\sigma,\rho)|
 \lesssim
 (1+|\ell_j|)
 (1+\omega_j|h_j|)^{-1/2}.
\end{equation}
We also have a scalar $L^2$ estimate. Since $\eta$ has fixed sign, make
the change $v=|\eta|^{r'}$. All factors except
$e^{i\omega_j\sigma t_j(x)^{r'}v}$ are unimodular or fixed amplitudes,
and Plancherel in $\sigma$ gives
\begin{align}\label{eq:L2d-J-L2}
 \int_{\R}|J_{j,\ell_j}(\sigma,\rho)|^2\,d\sigma
 &=
 \frac{C}{\omega_jt_j(x)^{r'}}
 \int
 \frac{|\vartheta_j(\eta)|^2}
 {r'|\eta|^{r'-1}}\,d\eta
 \lesssim\omega_j^{-1}.
\end{align}
In particular, this estimate is uniform in $\rho$, $x$, $y$, and
$\ell_j$.

Let $w(\sigma)=(1+|\sigma|)^{-2}$. Combining
\eqref{eq:L2d-J-sup} and \eqref{eq:L2d-J-L2}, and using $w\le1$, yields
\begin{equation}\label{eq:L2d-J-Ld}
 \left(
 \int_\R w(\sigma)|J_{j,\ell_j}(\sigma,\rho)|^d\,d\sigma
 \right)^{1/d}
 \lesssim
 (1+|\ell_j|)
 \omega_j^{-1/d}
 (1+\omega_j|h_j|)^{-(d-2)/(2d)} .
\end{equation}
For $j\notin\mathcal A$, the corresponding factor contains only the
linear phase $\omega_jh_j\eta$. Two integrations by parts give
$|J_{j,\ell_j}|\lesssim(1+|\ell_j|)^2
(1+\omega_j|h_j|)^{-2}$.

By Lemma~\ref{lem:scalar-kernel},
$R^2|k_R(R\sigma)k_R(R\rho)|\lesssim Rw(\sigma)w(\rho)$.
Since $|\mathcal A|=|S|+1\le d$, H\"older in $\sigma$ with exponent
$d$ for every $j\in\mathcal A$ gives
\[
 \int_\R w(\sigma)\prod_{j\in\mathcal A}|J_{j,\ell_j}|\,d\sigma
 \le
 \|w\|_1^{1-|\mathcal A|/d}
 \prod_{j\in\mathcal A}
 \left(\int_\R w(\sigma)|J_{j,\ell_j}|^d\,d\sigma\right)^{1/d}.
\]
Using \eqref{eq:L2d-J-Ld}, integrating $w(\rho)$, and summing the
Fourier series proves \eqref{eq:L2d-kernel-before-HLS}.
\end{proof}

\subsubsection{The \texorpdfstring{$L^2$}{L2} to \texorpdfstring{$L^{2d}$}{L(2d)} estimate for \texorpdfstring{$d\ge4$}{d >= 4}}

\begin{lemma}\label{lem:sector-L2d}
Let $d\ge4$. For some finite $M=M(d,r)$,
\begin{equation}\label{eq:sector-L2d}
 \|\mathcal R_{R,\delta,S,k}g\|_{L^{2d}(\R^d)}
 \lesssim
 \mathfrak B_M
 R^{1/2-1/d}
 \prod_{i\in S}\lambda_i^{-1/d}
 \prod_{\substack{i\le n\\i\notin S}}
 \omega_i^{-1/(2d)}
 \|g\|_2,
\end{equation}
uniformly in $R$, $\delta$, $S$, $k$, and the measurable function
$t=t(x)$.
\end{lemma}

\begin{proof}
Apply Lemma~\ref{lem:sector-kernel}. Since $d\ge4$, $(d-2)/(2d)\ge1/d$.
Hence
\[
 \begin{aligned}
 (1+\omega|h|)^{-(d-2)/(2d)}
 &\lesssim\omega^{-1/d}|h|^{-1/d},\\
 (1+\omega|h|)^{-2}
 &\lesssim\omega^{-1/d}|h|^{-1/d}.
 \end{aligned}
\]
As $\omega_d=R$ and $\omega_i=\lambda_i$ for $i\in S$,
\eqref{eq:L2d-kernel-before-HLS} becomes
\begin{equation}
 |K(x,y)|
 \lesssim
 \mathfrak B_M^2
 R^{1-2/d}
 \prod_{i\in S}\lambda_i^{-2/d}
 \prod_{\substack{i\le n\\i\notin S}}\omega_i^{-1/d}
 \prod_{j=1}^d|x_j-y_j|^{-1/d}.
\end{equation}

The one-dimensional Hardy--Littlewood--Sobolev inequality gives
\[
 \big\|\,|\cdot|^{-1/d}*F\,\big\|_{L^{2d}(\R)}
 \lesssim \|F\|_{L^{(2d)'}(\R)}.
\]
Iterating it in the $d$ coordinates, using Minkowski to reorder the
intermediate mixed norms
because $(2d)'\le2d$, yields
\begin{equation}\label{eq:L2d-product-HLS}
 \left\|
 \int_{\R^d}
 \prod_{j=1}^d|x_j-y_j|^{-1/d}F(y)\,dy
 \right\|_{2d}
 \lesssim
 \|F\|_{(2d)'}.
\end{equation}
Hence, for compactly supported $F$,
\begin{align*}
 \|\mathcal R_{R,\delta,S,k}^*F\|_2^2
 &=
 \langle
 \mathcal R_{R,\delta,S,k}
 \mathcal R_{R,\delta,S,k}^*F,F
 \rangle\\
 &\lesssim
 \mathfrak B_M^2
 R^{1-2/d}
 \prod_{i\in S}\lambda_i^{-2/d}
 \prod_{\substack{i\le n\\i\notin S}}\omega_i^{-1/d}
 \|F\|_{(2d)'}^2.
\end{align*}
Taking square roots and using duality gives
\[
 \|\mathcal R_{R,\delta,S,k}g\|_{2d}
 \lesssim
 \mathfrak B_M
 R^{1/2-1/d}
 \prod_{i\in S}\lambda_i^{-1/d}
 \prod_{\substack{i\le n\\i\notin S}}\omega_i^{-1/(2d)}
 \|g\|_2,
\]
which proves \eqref{eq:sector-L2d}.
\end{proof}

\subsubsection{The \texorpdfstring{$L^2$}{L2} to \texorpdfstring{$L^{12}$}{L12} estimate in dimension three}

In dimension three the same kernel argument gives a different improving
estimate, which will be used only in conjunction with local smoothing.

\begin{lemma}\label{lem:sector-L12}
Let $d=3$. For some finite $M=M(r)$,
\begin{equation}\label{eq:sector-L12}
 \|\mathcal R_{R,\delta,S,k}g\|_{L^{12}(\R^3)}
 \lesssim
 \mathfrak B_M R^{1/4}
 \prod_{i\in S}\lambda_i^{-1/4}
 \prod_{\substack{i\le2\\i\notin S}}\omega_i^{-1/12}
 \|g\|_2,
\end{equation}
uniformly in $R$, $\delta$, $S$, $k$, and the measurable selector
$t=t(x)$.
\end{lemma}

\begin{proof}
Apply Lemma~\ref{lem:sector-kernel} with $d=3$, and put
$\mathcal A=S\cup\{3\}$. Its kernel bound reads
\[
 |K(x,y)|\lesssim
 \mathfrak B_M^2 R
 \prod_{j\in\mathcal A}
 \omega_j^{-1/3}(1+\omega_j|x_j-y_j|)^{-1/6}
 \prod_{\substack{j\le2\\j\notin S}}
 (1+\omega_j|x_j-y_j|)^{-2}.
\]
Using $(1+\omega|h|)^{-1/6}\le\omega^{-1/6}|h|^{-1/6}$
and the corresponding bound for the factors of exponent $-2$, we get
\[
 |K(x,y)|\lesssim
 \mathfrak B_M^2 R^{1/2}
 \prod_{i\in S}\lambda_i^{-1/2}
 \prod_{\substack{i\le2\\i\notin S}}\omega_i^{-1/6}
 \prod_{j=1}^3|x_j-y_j|^{-1/6}.
\]
The one-dimensional Hardy--Littlewood--Sobolev inequality gives
\[
 \big\|\,|\cdot|^{-1/6}*F\,\big\|_{L^{12}(\R)}
 \lesssim \|F\|_{L^{12/11}(\R)}.
\]
Iteration in the three coordinates, with Minkowski's inequality as
in \eqref{eq:L2d-product-HLS}, yields the corresponding product-kernel
bound from $L^{12/11}(\R^3)$ to $L^{12}(\R^3)$. Applying this bound
to $\mathcal R\mathcal R^*$, then taking square roots and using
duality, proves \eqref{eq:sector-L12}.
\end{proof}

\subsection{Sector summation and spatial localization}

We now restore the sector factors and sum over $(S,k)$. We first
consider the $L^4$ estimate. Fix $x_d=c$ and an arbitrary measurable
selector $t=t(x',c)$. Put $J_R=R\prod_{i=1}^n\omega_i$ and
$g(\eta)=\widehat f(\mathfrak D\eta)$. Applying Cauchy--Schwarz in
the scalar variable and Lemma~\ref{lem:fixed-height-L4} to the reduced
operator gives
\begin{equation}\label{eq:fixed-height-sector}
 \|A_{\delta,S,k}P_R^{(d)}f
 (\cdot,c,t(\cdot,c))\|_{L^4_{x'}}
 \lesssim
 \mathfrak B_M C_{R,\delta,S,k}
 R^{1/2}\prod_{i=1}^n\omega_i^{1/4}\|f\|_2,
\end{equation}
where we used $\|g\|_2=J_R^{-1/2}\|f\|_2$.

Recall that
$C_{R,\delta,S,k}
=
V_\delta
\prod_{i\in S}\lambda_i^{-1/2}
\prod_{\substack{i\in I\setminus S\\k_i=0}}\lambda_i^{-K}
\prod_{i:k_i\ge1}(2^{k_i}\lambda_i)^{-K}$.
Taking $K=1$, an index $i\in S$ contributes
$\lambda_i^{-1/2}\omega_i^{1/4}=\lambda_i^{-1/4}$, while
$i\in I\setminus S$ with $k_i=0$ contributes
$\lambda_i^{-3/4}\le\lambda_i^{-1/4}$. For the remote sectors,
$\sum_{k_i\ge1}(2^{k_i}\lambda_i)^{-3/4}
\lesssim\lambda_i^{-3/4}\le\lambda_i^{-1/4}$, and if $i\notin I$ then
$\lambda_i=1$ and the same sum is uniformly bounded. Hence
\begin{equation}\label{eq:sector-sum-L4}
 \sum_{S,k}
 C_{R,\delta,S,k}\prod_{i=1}^n\omega_i^{1/4}
 \lesssim
 V_\delta\prod_{i=1}^n\lambda_i^{-1/4}.
\end{equation}
Linearizing first over a countable dense subset of $[1,2]^d$ and then
using continuity in $t$, \eqref{eq:fixed-height-sector} and
\eqref{eq:sector-sum-L4} give, uniformly in $c$,
\begin{equation}\label{eq:fixed-height-patch}
 \left\|
 \sup_{t\in[1,2]^d}
 |A_\delta P_R^{(d)}f(\cdot,c,t)|
 \right\|_{L^4_{x'}}
 \lesssim
 \mathfrak B_M
 R^{1/2}V_\delta
 \prod_{i=1}^n\lambda_i^{-1/4}\|f\|_2.
\end{equation}

Before passing from the fixed-height bound to the full $L^4$
estimate, we sum the $L^{2d}$ sectors for $d\ge4$ and the $L^{12}$
sectors for $d=3$. For $d\ge4$ and an arbitrary
measurable selector $t=t(x)$, \eqref{reduced-oio-2d} is
$C_{R,\delta,S,k}J_R\mathcal R_{R,\delta,S,k}g$. Hence
Lemma~\ref{lem:sector-L2d} and $\|g\|_2=J_R^{-1/2}\|f\|_2$ give
\begin{align}\label{eq:L2d-sector-recombined}
 \|A_{\delta,S,k}P_R^{(d)}f(\cdot,t(\cdot))\|_{2d}
 \lesssim{}&
 \mathfrak B_M C_{R,\delta,S,k}
 R^{(d-1)/d}
 \prod_{i\in S}\lambda_i^{1/2-1/d}
 \nonumber\\
 &\times
 \prod_{\substack{i\le n\\i\notin S}}
 \omega_i^{(d-1)/(2d)}
 \|f\|_2.
\end{align}
For $i\in S$, the corresponding factor in
$C_{R,\delta,S,k}$ gives
$\lambda_i^{-1/2}\lambda_i^{1/2-1/d}=\lambda_i^{-1/d}$.
For $i\notin S$, put $\gamma=1-(d-1)/(2d)=(d+1)/(2d)>1/d$.
If $i\in I\setminus S$ and $k_i=0$, its contribution is
$\lambda_i^{-\gamma}\le\lambda_i^{-1/d}$, while
\[
 \sum_{k_i\ge1}
 (2^{k_i}\lambda_i)^{-\gamma}
 \lesssim
 \lambda_i^{-\gamma}
 \le
 \lambda_i^{-1/d}.
\]
Again, when $i\notin I$ one has $\lambda_i=1$, and the corresponding
bounded and remote contributions are uniformly summable. Therefore
\begin{equation}\label{eq:sector-sum-L2d}
 \sum_{S,k}
 C_{R,\delta,S,k}
 \prod_{i\in S}\lambda_i^{1/2-1/d}
 \prod_{\substack{i\le n\\i\notin S}}
 \omega_i^{(d-1)/(2d)}
 \lesssim
 V_\delta\prod_{i=1}^n\lambda_i^{-1/d}.
\end{equation}
Since the selector is arbitrary, \eqref{eq:L2d-sector-recombined} and
\eqref{eq:sector-sum-L2d} yield
\begin{equation}
 \|M_{\loc}^\delta P_R^{(d)}f\|_{2d}
 \lesssim
 \mathfrak B_M
 R^{(d-1)/d}V_\delta
 \prod_{i=1}^n\lambda_i^{-1/d}\|f\|_2.
\end{equation}
This proves \eqref{eq:normalized-L2d-bound}.

For $d=3$, we similarly recombine Lemma~\ref{lem:sector-L12}.
Since $J_R=R\omega_1\omega_2$ and
$\|g\|_2=J_R^{-1/2}\|f\|_2$, it gives
\begin{align*}
 \|A_{\delta,S,k}P_R^{(3)}f(\cdot,t(\cdot))\|_{12}
 \lesssim{}&
 \mathfrak B_M C_{R,\delta,S,k}R^{3/4}
 \prod_{i\in S}\lambda_i^{1/4}
 \prod_{\substack{i\le2\\i\notin S}}\omega_i^{5/12}
 \|f\|_2.
\end{align*}
An active stationary coordinate contributes
$\lambda_i^{-1/2}\lambda_i^{1/4}=\lambda_i^{-1/4}$.
Taking $K=1$ in $C_{R,\delta,S,k}$, a nonstationary active coordinate
with $k_i=0$ contributes
$\lambda_i^{-1}\lambda_i^{5/12}
 =\lambda_i^{-7/12}\le\lambda_i^{-1/4}$.
The remote sectors satisfy
\[
 \sum_{k_i\ge1}(2^{k_i}\lambda_i)^{-7/12}
 \lesssim\lambda_i^{-7/12}\le\lambda_i^{-1/4}.
\]
When $i\notin I$, $\lambda_i=1$ and the corresponding bounded and
remote contributions are again uniformly summable. Consequently,
\begin{equation}
 \|M_{\loc}^\delta P_R^{(3)}f\|_{12}
 \lesssim
 \mathfrak B_M R^{3/4}V_\delta
 (\lambda_1\lambda_2)^{-1/4}\|f\|_2.
\end{equation}
This proves \eqref{eq:normalized-L12-bound}. By
\eqref{eq:exact-spatial-rescaling}, the power of $V_\delta$ becomes
$1+1/12-1/2=7/12$ in the original spatial coordinates, which gives
\eqref{eq:dyadic-2-12}.

It remains to pass from the fixed-height estimate
\eqref{eq:fixed-height-patch} to the full $L^4$ norm. The kernel of
$A_\delta P_R^{(d)}$ is
\begin{equation}
 K_{t,R,\delta}(z)
 =
 V_\delta
 \left(\prod_{i=1}^nt_i^{-1}\right)
 B_\delta\left(\frac{z_1}{t_1},\ldots,\frac{z_n}{t_n}\right)
 R\check\psi\left(
 R\left[
 z_d-t_d\Phi_\delta
 \left(\frac{z_1}{t_1},\ldots,\frac{z_n}{t_n}\right)
 \right]\right).
\end{equation}
Since $B_\delta$ has fixed compact support and $\Phi_\delta$ is
uniformly bounded there, there are fixed $H,C>0$ such that, for every
$N$ and $|z_d|\ge H$,
\[
 \sup_{t\in[1,2]^d}|K_{t,R,\delta}(z)|
 \lesssim_N
 \mathfrak B_0V_\delta R^{1-N}
 {\bf1}_{\{|z'|\le C\}}(1+|z_d|)^{-N}.
\]
Let $E_R$ denote the right-hand side with the additional factor
${\bf1}_{\{|z_d|\ge H\}}$; then
\[
 \|E_R\|_1+\|E_R\|_{4/3}
 \lesssim_N\mathfrak B_0V_\delta R^{1-N}.
\]

Fix $H_1>H$ and set
$f_c(y',y_d)={\bf1}_{\{|y_d-c|\le H_1\}}f(y',y_d)$. At $x_d=c$,
\[
 \sup_t|A_\delta P_R^{(d)}f(x',c,t)|
 \le
 \sup_t|A_\delta P_R^{(d)}f_c(x',c,t)|
 +(E_R*|f|)(x',c).
\]
The fixed-height estimate extends to $L^2$ inputs by density.
Applying \eqref{eq:fixed-height-patch} to $f_c$ gives
\[
 \left\|
 \sup_t|A_\delta P_R^{(d)}f_c(\cdot,c,t)|
 \right\|_{L^4_{x'}}
 \lesssim
 \mathfrak B_M
 R^{1/2}V_\delta
 \prod_{i=1}^n\lambda_i^{-1/4}\|f_c\|_2.
\]
Writing $H_f(v)=\int_{\R^n}|f(y',v)|^2\,dy'$, we have
$\|f_c\|_2^2=({\bf1}_{[-H_1,H_1]}*H_f)(c)$, and Young's inequality gives
\[
 \int_\R\|f_c\|_2^4\,dc
 =
 \|{\bf1}_{[-H_1,H_1]}*H_f\|_2^2
 \lesssim
 \|H_f\|_1^2
 =
 \|f\|_2^4.
\]
Consequently,
\begin{equation}\label{eq:global-L4-near}
 \left(
 \int_\R
 \left\|
 \sup_t|A_\delta P_R^{(d)}f_c(\cdot,c,t)|
 \right\|_{L^4_{x'}}^4\,dc
 \right)^{1/4}
 \lesssim
 \mathfrak B_M
 R^{1/2}V_\delta
 \prod_{i=1}^n\lambda_i^{-1/4}\|f\|_2.
\end{equation}

For the far part, Young's inequality gives
$\|E_R*|f|\|_4\lesssim_N
\mathfrak B_0V_\delta R^{1-N}\|f\|_2$. Since
$\lambda_i\le R$,
$R^{1/2}V_\delta\prod_i\lambda_i^{-1/4}
\ge V_\delta R^{-(d-3)/4}$.
Choosing $N>1+(d-3)/4$ therefore absorbs the far term into
\eqref{eq:global-L4-near}. Hence
\begin{equation}
 \|M_{\loc}^\delta P_R^{(d)}f\|_4
 \lesssim
 \mathfrak B_M
 R^{1/2}V_\delta
 \prod_{i=1}^n\lambda_i^{-1/4}\|f\|_2,
\end{equation}
which proves \eqref{eq:normalized-L4-bound}.

\subsection{The normalized \texorpdfstring{$L^2$}{L2} estimate}
\label{sec:normalized-L2}

We deduce \eqref{eq:normalized-L2-bound} from
\eqref{eq:normalized-L4-bound} by spatial localization. Write the kernel
as $K_{t,R,\delta}=K_{t,R,\delta}^{\mathrm{near}}
+K_{t,R,\delta}^{\mathrm{far}}$, where
$\supp K_{t,R,\delta}^{\mathrm{near}}$ is contained in a fixed bounded
set, uniformly in $t$, $R$, and $\delta$, while
$\sup_t|K_{t,R,\delta}^{\mathrm{far}}|\le E_R$ with
$\|E_R\|_1\lesssim_N\mathfrak B_0V_\delta R^{1-N}$.

Let $\{Q_a\}$ be the unit-cube decomposition of $\R^d$ and choose fixed
enlargements $Q_a^*$ so that
$Q_a-\supp K_{t,R,\delta}^{\mathrm{near}}\subset Q_a^*$ for every $t$.
With $f_a={\bf1}_{Q_a^*}f$, for $x\in Q_a$ we have
\[
 M_{\loc}^\delta P_R^{(d)}f(x)
 \le
 M_{\loc}^\delta P_R^{(d)}f_a(x)
 +(E_R*|f|)(x).
\]
Since $|Q_a|=1$, H\"older's inequality and
\eqref{eq:normalized-L4-bound} give
\[
 \|{\bf1}_{Q_a}M_{\loc}^\delta P_R^{(d)}f_a\|_2
 \lesssim\mathfrak B_M b_{\delta,4}(R)\|f_a\|_2.
\]
The sets $Q_a^*$ have bounded overlap, and hence
\[
 \|M_{\loc}^\delta P_R^{(d)}f\|_2^2
 \lesssim
 \mathfrak B_M^2b_{\delta,4}(R)^2\|f\|_2^2
 +\|E_R*|f|\|_2^2.
\]
By Young's inequality,
$\|E_R*|f|\|_2\lesssim_N
\mathfrak B_0V_\delta R^{1-N}\|f\|_2$. Since
$\lambda_i\le R$, we have
\[
 b_{\delta,4}(R)=R^{1/2}V_\delta\prod_{i=1}^n\lambda_i^{-1/4}
 \ge V_\delta R^{-(d-3)/4}.
\]
Thus, choosing $N>1+(d-3)/4$, the far term is absorbed and
\[
 \|M_{\loc}^\delta P_R^{(d)}f\|_2
 \lesssim
 \mathfrak B_M b_{\delta,4}(R)\|f\|_2.
\]
This proves \eqref{eq:normalized-L2-bound}.

\begin{remark}\label{rem:shifted-quadratic-model}
For comparison, let $n=d-1\ge3$, fix $c\ne0$ and
$b\in C_c^\infty(\mathbb R^n)$, and define
\[
 \mathcal Q_{\loc}f(x)
 =\sup_{t\in[1,2]^d}
 \left|\int_{\mathbb R^n}
 f\bigl(x-(t_1v_1,\ldots,t_nv_n,t_d(c+|v|^2))\bigr)
 b(v)\,dv\right|.
\]
Then
\[
 \|\mathcal Q_{\loc}f\|_q\lesssim_{p,q,b,c}\|f\|_p,
 \qquad p>\frac{n+2}{n},\quad p\le q<np.
\]
To see this, localize $\xi_d$ to $|\xi_d|\asymp R$ and write
$\xi=R\eta$. Completing the square in $v$ and applying
stationary phase gives a multiplier of the form
\[
 R^{-n/2}
 e^{-iRc t_d\eta_d+
 iR\sum_{i=1}^n t_i^2\eta_i^2/(4t_d\eta_d)}a_R(t,\eta),
\]
where $a_R$ has uniformly bounded derivatives on a fixed
frequency box. Frequencies $|\xi'|\gg R$ give rapidly
decreasing nonstationary pieces and can be summed separately.
For each fixed $\eta_d$, the tangential $TT^*$ kernel is bounded
by a constant times
$\prod_{i=1}^n(1+R|x_i-y_i|)^{-1/2}$, uniformly in arbitrary
measurable selectors $t=t(x)$. This follows from the
linear--quadratic version of Lemma~\ref{lem:mixed-linear-power}
in each tangential variable. In the quadratic case the second
derivative is constant, so the same estimate holds for cutoffs
whose support meets zero. The factors involving $x_d$ and $c$ have
modulus one at each fixed $\eta_d$ and do not change these norms.
Product Hardy--Littlewood--Sobolev estimates give tangential
$L^2\to L^4$ and $L^2\to L^{2n}$ norms bounded by
$R^{-n/4}$ and $R^{-1/2}$, respectively. Integrating in the
compact $\eta_d$ interval by Minkowski and Cauchy--Schwarz,
and restoring the factor $R^{d/2-n/2}=R^{1/2}$, gives
fixed-height norms $R^{-(n-2)/4}$ and $1$.
The spatial localization used above therefore gives
\[
 \begin{aligned}
 \|\mathcal Q_{\loc}P_R^{(d)}f\|_2
 +\|\mathcal Q_{\loc}P_R^{(d)}f\|_4
 &\lesssim R^{-(n-2)/4}\|f\|_2,\\
 \|\mathcal Q_{\loc}P_R^{(d)}f\|_{2n}
 &\lesssim\|f\|_2.
 \end{aligned}
\]
The elementary kernel bounds also give the norms $R,R,1$ for
$L^1\to L^1$, $L^1\to L^\infty$, and
$L^\infty\to L^\infty$, respectively.

Set $u=1/p$ and $v=1/q$, and let $a$ denote the exponent of $R$.
Put $\kappa=(n-2)/4$ and $U=n/(n+2)$.
The interpolation points include
\[
 O=(0,0,0),\quad E=(1/2,1/(2n),0),\quad
 A=(1/2,1/4,-\kappa),\quad D=(1/2,1/2,-\kappa),
\]
and $I=(1,0,1)$, $J=(1,1,1)$. With
$\theta=(n-2)/(n+2)$, the convex combinations
\[
 Z=\theta I+(1-\theta)A=(U,U/n,0),\qquad
 Z_D=\theta J+(1-\theta)D=(U,U,0)
\]
have zero frequency exponent. The segments $OE$ and $EZ$
give zero exponents on $v=u/n$ for $0<u<U$; the segments
$OD$ and $DZ_D$ give strictly negative exponents on $v=u$
over the same interval. Interpolation at fixed $u$ gives a
strictly negative exponent whenever $u/n<v\le u$.
Dyadic summation in $R$, together with the low-frequency
kernel envelope, proves the stated sufficient range.
This argument does not determine the boundary $q=np$.
\end{remark}

\section{Exact factorization by iterated stationary phase}
\label{sec:proof-lem-exact}

Throughout this proof we work on the support of $\psi(\xi_d/R)$ and
on the positive frequency branch fixed above. We prove
\eqref{eq:sector-factorization} by eliminating the integration
variables in a specified order. We first eliminate the remote and
active nonstationary variables, then the active stationary variables,
and finally the bounded inactive variables. All auxiliary derivatives
below are taken in the parameters, normalized frequencies, and
remaining integration variables, while the scales are held fixed.

\subsection{Nonstationary variables}

We first eliminate the nonstationary variables. These are either
$i\in I\setminus S$ with $k_i=0$, or $k_i\ge1$. Fix such an index $i$
and put $\varepsilon_i=\delta_i^r$. With the previously eliminated
coordinates omitted, write
\[
 A=1-\sum_{j\ne i}(\delta_jv_j)^r.
\]
We split the phase into the part independent of $v_i$
\[
 \Phi_0
 =
 \sum_{j\ne i}t_j\xi_jv_j+t_d\xi_dA^{1/r}
\]
and
\begin{equation}\label{eq:local-root-difference}
 \Phi_i
 =
 t_i\xi_iv_i
 +
 t_d\xi_d
 \left[
 (A-\varepsilon_iv_i^r)^{1/r}-A^{1/r}
 \right].
\end{equation}
The identity
\begin{equation}\label{eq:local-divided-difference}
 \frac{(A-\varepsilon v^r)^{1/r}-A^{1/r}}{\varepsilon}
 =
 -\frac{v^r}{r}
 \int_0^1
 (A-s\varepsilon v^r)^{1/r-1}\dd s
\end{equation}
shows that every fixed-order derivative of the second term in
\eqref{eq:local-root-difference} is $O(R\delta_i^r)$ on the support.

After a finite smooth refinement of the frequency cutoffs, with the
constants in the stationary and nonstationary sectors chosen uniformly,
the definition of the nonstationary sectors gives
\[
 |\partial_{v_i}\Phi_i|
 \gtrsim
 \begin{cases}
 R\delta_i^r,&i\in I\setminus S,\ k_i=0,\\
 \omega_i,&k_i\ge1,
 \end{cases}
 \qquad
 \omega_i=2^{k_i}\lambda_i.
\]
All derivatives of the corresponding normalized phases are uniformly
bounded. Therefore repeated integration by parts in $v_i$ gives, for
every fixed $K,N'\ge1$,
\[
 \int e^{-i\Phi_i}B_\delta\,\dd v_i
 =
 \begin{cases}
 (R\delta_i^r)^{-K}\widetilde B_\delta,
   &i\in I\setminus S,\ k_i=0,\\[1mm]
 (2^{k_i}\lambda_i)^{-K}\widetilde B_\delta,
   &k_i\ge1,
 \end{cases}
\]
where
\[
 \|\widetilde B_\delta\|_{C^{N'}}
 \le
 C_{K,N'}\|B_\delta\|_{C^{M(K,N')}}.
\]
Indeed, if $N'$ auxiliary derivatives are required after the
integration by parts, it is enough to perform sufficiently many
additional integrations by parts, since differentiation of the
oscillatory factor produces at most the corresponding frequency scale
to a fixed finite power.

Eliminating all nonstationary coordinates therefore contributes
\begin{equation}\label{conclusion-non-stat2}
 \prod_{\substack{i\in I\setminus S\\k_i=0}}
 (R\delta_i^r)^{-K}
 \prod_{i:k_i\ge1}
 (2^{k_i}\lambda_i)^{-K}.
\end{equation}

The remaining coordinates are the stationary indices $S$ and the
bounded inactive indices
\[
 J_0=\{i:i\notin I,\ k_i=0\}.
\]
Set
\[
 J=S\cup J_0,
 \qquad
 A_J=1-\sum_{j\in J}(\delta_jv_j)^r.
\]
After the nonstationary variables have been eliminated, the remaining
phase is
\[
 \Phi_J
 =
 \sum_{j\in J}t_j\xi_jv_j+t_d\xi_dA_J^{1/r}.
\]

\subsection{Stationary variables}

We now eliminate the variables $v_i$ with $i\in S$. Consider
\begin{equation}
 \int_{\mathbb R^{|J|}}
 e^{-i\Phi_J}
 \widetilde B_\delta(v_J)\dd v_J.
\end{equation}
Fix $i\in S$ and put
\[
 A_i
 =
 1-\sum_{j\in J\setminus\{i\}}(\delta_jv_j)^r.
\]
Then
\[
 \Phi_J=\Phi_0+\Psi_i,
\]
where
\[
 \Phi_0
 =
 \sum_{j\in J\setminus\{i\}}t_j\xi_jv_j
 +
 t_d\xi_dA_i^{1/r}
\]
and
\begin{equation}\label{stat-int-2}
 \Psi_i
 =
 t_i\xi_iv_i
 +
 t_d\xi_d
 \left[
 (A_i-(\delta_iv_i)^r)^{1/r}-A_i^{1/r}
 \right].
\end{equation}

Since $i\in S$, one has
$|\xi_i|\asymp R\delta_i^r$, while $|\xi_d|\asymp R$.
On the fixed stationary sign sector put
\[
 a=t_i\frac{\xi_i}{R\delta_i^r},
 \qquad
 b=t_d\frac{\xi_d}{R},
 \qquad
 \varepsilon=\delta_i^r.
\]
Thus $a$ and $b$ stay in fixed compact subsets of $(0,\infty)$.
Dividing \eqref{stat-int-2} by $R\delta_i^r$ gives
\[
 G(v)
 =
 av
 +
 b\frac{(A-\varepsilon v^r)^{1/r}-A^{1/r}}{\varepsilon},
 \qquad A=A_i.
\]
A direct differentiation gives
\[
 G''(v)
 =
 -b(r-1)A\,v^{r-2}
 (A-\varepsilon v^r)^{1/r-2}.
\]
After a finite smooth refinement, the normalized parameters range
over a fixed compact set. On a common enlargement of the $v$-support,
\[
 |G''(v)|\ge c_0>0,
\]
and all required derivatives of $G$ are uniformly bounded.
These bounds apply to both the near-critical and nonstationary parts
treated below.
Let $r'=r/(r-1)$. The critical point $v_*$ and the corresponding
graph height $h_*=(A-\varepsilon v_*^r)^{1/r}$ satisfy
\[
 \frac{v_*}{h_*}
 =
 \left(\frac ab\right)^{r'-1},
 \qquad
 h_*^r+\varepsilon v_*^r=A.
\]
Hence
\begin{equation}
 h_*
 =
 \frac{A^{1/r}b^{r'-1}}
 {(b^{r'}+\varepsilon a^{r'})^{1/r}},
 \qquad
 v_*
 =
 \frac{A^{1/r}a^{r'-1}}
 {(b^{r'}+\varepsilon a^{r'})^{1/r}}.
\end{equation}
Consequently,
\begin{align}
 R\delta_i^r av_*+Rbh_*
 &=
 R A^{1/r}
 (b^{r'}+\varepsilon a^{r'})^{1/r'}
 \nonumber\\
 &=
 A^{1/r}
 \left(
 |t_d\xi_d|^{r'}
 +
 \left|\frac{t_i\xi_i}{\delta_i}\right|^{r'}
 \right)^{1/r'}.
\end{align}
Combining this critical value with the $v_i$-independent phase shows that the
coefficient of the remaining root $A_i^{1/r}$ is replaced by
\[
 \left(
 |t_d\xi_d|^{r'}
 +
 \left|\frac{t_i\xi_i}{\delta_i}\right|^{r'}
 \right)^{1/r'}.
\]

Put $\mu=R\delta_i^r$, with all scales held fixed. After smooth
localizations, the part near $v_*$ admits the Morse coordinate
\[
 w=(v-v_*)\left(2\int_0^1(1-s)
 [-G''(v_*+s(v-v_*))]\dd s\right)^{1/2},
\]
so that $G(v)-G(v_*)=-w^2/2$. The coordinate change and its required
auxiliary derivatives are uniformly bounded. Write
$\widetilde B_\delta^{\mathrm{near}}$ for the localized amplitude and
$\mathfrak b(w)$ for its transform under this change of variables.
The Gaussian Fourier identity gives the exact formula
\[
 \int e^{-i\mu G(v)}\widetilde B_\delta^{\mathrm{near}}(v)\dd v
 =c\mu^{-1/2}e^{-i\mu G(v_*)}
 \int e^{-i\zeta^2/(2\mu)}\widehat{\mathfrak b}(\zeta)\dd\zeta,
\]
where $c$ is a fixed Fourier normalization constant. Auxiliary
derivatives act only on $\mathfrak b$, and the Fourier $L^1$ norms of these
derivatives are controlled by finitely many derivatives of the incoming
amplitude. Thus the normalized amplitude has uniform finite-order
derivative bounds. The complementary part is uniformly nonstationary.
After factoring $e^{-i\mu G(v_*)}$, sufficiently many additional
integrations by parts compensate for each prescribed auxiliary
derivative of this factor. This also covers critical points outside
the support enlargement and yields the same exact factorization and
symbol bounds.

The same calculation can now be iterated. If $T\subset S$ has already
been eliminated, the coefficient of the remaining root is
\[
 H_T
 =
 \left(
 |t_d\xi_d|^{r'}
 +
 \sum_{i\in T}
 \left|\frac{t_i\xi_i}{\delta_i}\right|^{r'}
 \right)^{1/r'}.
\]
Since $\xi_d=R\eta_d$ and, for $i\in S$,
$\xi_i=R\delta_i^r\eta_i$, one has
\[
 \frac{H_T}{R}
 =
 \left(
 t_d^{r'}|\eta_d|^{r'}
 +
 \sum_{i\in T}
 \delta_i^r t_i^{r'}|\eta_i|^{r'}
 \right)^{1/r'}.
\]
This quantity and all required normalized derivatives are uniformly
bounded, and it has a fixed positive lower bound. Thus the same
stationary phase argument may be applied successively to every
$i\in S$.

After eliminating all stationary variables, we obtain
\begin{align}
 &\int_{\mathbb R^{|J|}}
 e^{-i\Phi_J}\widetilde B_\delta(v_J)\dd v_J
 \nonumber\\
 &\qquad=
 \left(
 \prod_{i\in S}(R\delta_i^r)^{-1/2}
 \right)
 \nonumber\\
 &\qquad\quad\times
 \int_{\mathbb R^{|J_0|}}
 e^{-i\left[
 \sum_{j\in J_0}t_j\xi_jv_j
 +
 h_S(t,\eta)
 \left(1-\sum_{j\in J_0}(\delta_jv_j)^r\right)^{1/r}
 \right]}
 a^{(1)}_{R,\delta,S,k}(t,\eta,v_{J_0})
 \dd v_{J_0},
 \label{conclusion-stat}
\end{align}
where
\[
 h_S(t,\eta)
 =
 \left(
 |t_d\xi_d|^{r'}
 +
 \sum_{i\in S}
 \left|\frac{t_i\xi_i}{\delta_i}\right|^{r'}
 \right)^{1/r'}.
\]
Equivalently,
\[
 h_S(t,\eta)
 =
 R\left\{
 t_d^{r'}
 \left[
 |\eta_d|^{r'}
 +
 \sum_{i\in S}
 \delta_i^r
 \left(\frac{t_i}{t_d}\right)^{r'}
 |\eta_i|^{r'}
 \right]
 \right\}^{1/r'}.
\]

\subsection{Bounded inactive variables}

It remains to integrate over $J_0$. For $j\in J_0$ one has
$R\delta_j^r<1$, $\lambda_j=1$, $k_j=0$, and $|\xi_j|\lesssim1$.
Moreover, $h_S/R$ and its required normalized derivatives are
uniformly bounded.

Factor $e^{-ih_S(t,\eta)}$ from the integrand in
\eqref{conclusion-stat} and define
\begin{align*}
 a_{R,\delta,S,k}(t,\eta)
 :=
 \int_{\mathbb R^{|J_0|}}
 &e^{-i\left[
 \sum_{j\in J_0}t_j\xi_jv_j
 +
 h_S(t,\eta)
 \left\{
 \left(1-\sum_{j\in J_0}(\delta_jv_j)^r\right)^{1/r}-1
 \right\}
 \right]}
 \\
 &\times
 a^{(1)}_{R,\delta,S,k}(t,\eta,v_{J_0})
 \dd v_{J_0}.
\end{align*}
By \eqref{eq:local-divided-difference},
\[
 \left|
 h_S(t,\eta)
 \left[
 \left(1-\sum_{j\in J_0}(\delta_jv_j)^r\right)^{1/r}-1
 \right]
 \right|
 \lesssim
 R\sum_{j\in J_0}\delta_j^r
 \lesssim
 |J_0|,
\]
and the same uniform bound holds for every required normalized
derivative. The linear terms are also uniformly bounded because
$|\xi_j|\lesssim1$. Hence the residual exponential may be absorbed
into the amplitude and the $J_0$ variables may be integrated
absolutely. Therefore
\begin{align}\label{conclusion-bdd-inactive}
 &\int_{\mathbb R^{|J_0|}}
 e^{-i\left[
 \sum_{j\in J_0}t_j\xi_jv_j
 +
 h_S(t,\eta)
 \left(1-\sum_{j\in J_0}(\delta_jv_j)^r\right)^{1/r}
 \right]}
 a^{(1)}_{R,\delta,S,k}
 \dd v_{J_0}
 \nonumber\\
 &\qquad=
 e^{-ih_S(t,\eta)}
 a_{R,\delta,S,k}(t,\eta).
\end{align}

Working backwards through the at most $n$ elimination steps,
we obtain a finite integer $M(K,N,d)$ such that
\[
 \|a_{R,\delta,S,k}\|_{C_{(t,\eta)}^N}
 \le
 C_{K,N}
 \|B_\delta\|_{C^{M(K,N,d)}}.
\]
Finally, restoring the sector cutoffs and the factor $V_\delta$, and
combining \eqref{conclusion-non-stat2},
\eqref{conclusion-stat}, and
\eqref{conclusion-bdd-inactive}, each sector of $m_{t,\delta}$ is of
the form
\[
 V_\delta
 \prod_{i\in S}(R\delta_i^r)^{-1/2}
 \prod_{\substack{i\in I\setminus S\\k_i=0}}
 (R\delta_i^r)^{-K}
 \prod_{i:k_i\ge1}
 (2^{k_i}\lambda_i)^{-K}
 e^{-ih_S(t,\eta)}
 a_{R,\delta,S,k}(t,\eta).
\]
This proves \eqref{eq:sector-factorization} and
\eqref{eq:local-sector-amplitude}. The identity
\eqref{eq:local-dual-phase} follows from
$\xi_i=R\delta_i^r\eta_i$ for $i\in S$ and
$(r-1)r'=r$.

\section{The global theorem and sharpness}
\label{sec:global-theorem}

We extend the local estimates to the global maximal function $M$ and
then prove the necessity of the exponent range.
The dyadic decomposition of the dilation parameters, the coordinatewise
low--high frequency splitting, and the rescaling argument follow the
approach of Lee, Lee, and Oh~\cite[Section~3]{LLO2025-2}.
We formulate the local-to-global implication for compactly supported
positive measures under the coordinatewise decay hypothesis
\eqref{ltg:eq:hyp}, so that it also applies to their coordinate
projections.

For $p\ge2$, the high-frequency summation follows the same approach.
For $1<p<2$, we keep the spatial exponent $p$ fixed and interpolate
between $L^p(\ell^p)$ and $L^p(\ell^\infty)$ estimates to obtain an
$L^p(\ell^2)$ bound. A bootstrap argument then gives uniform bounds
for the maximal operators restricted to finitely many dilation boxes.
This treatment of $p<2$ is also needed in the Euclidean case and
is independent of the curvature degeneracy when $r>2$.

Let $\mu$ be a finite positive Borel measure on $\R^d$ with
$\supp\mu\subset[-c,c]^d$ for some $c<\infty$. Recall that
\[
 A f(x,t)=\int f(x-\mathrm D_t y)\dd\mu(y),\qquad
 \mathrm D_t=\diag(t_1,\ldots,t_d),
\]
and define
\[
 M_\loc f(x)=\sup_{t\in[1,2]^d}|Af(x,t)|,
 \qquad Mf(x)=\sup_{t\in(0,\infty)^d}|Af(x,t)|.
\]
We first work with Schwartz functions. Their averages are continuous
in $t$, so every supremum can equivalently be taken over a countable
dense parameter set.

Use the cutoffs $\psi_0$, $\psi$ and the coordinate projections
from Section~\ref{sec:coordinate-decay}. We extend $P_m^{(i)}$ to
all $m\in\Z$ by the same formula
$\widehat{P_m^{(i)}f}(\xi)=\psi(2^{-m}\xi_i)\widehat f(\xi)$.
In particular, $I=L_0^{(i)}+\sum_{m\ge1}P_m^{(i)}$.
The projections have uniformly bounded $L^p$ operator norms because
their kernels are dilates of a fixed Schwartz function.

\begin{proposition}\label{ltg:thm:main}
Let $1<p<\infty$. Suppose there exist $C_0<\infty$ and $\epsilon>0$
such that, for every $i=1,\ldots,d$ and $m\ge1$,
\begin{equation}\label{ltg:eq:hyp}
 \left\|M_\loc P_m^{(i)}f\right\|_p
 \le C_0 2^{-\epsilon m}\left\|f\right\|_p.
\end{equation}
Then
\begin{equation}
 \left\|M f\right\|_p\le C\left\|f\right\|_p.
\end{equation}
The constant depends only on $d,p,\epsilon,C_0,c,\mu(\R^d)$,
and the chosen cutoffs. No symmetry or curvature assumption on $\mu$
is needed once \eqref{ltg:eq:hyp} is available.
\end{proposition}
For $d\ge4$, the high-frequency argument proving
Theorem~\ref{thm:lplq}, with $p=q$, gives decay in the distinguished
frequency coordinate whenever $p>p_c(d,r)$. For $d=3$, this decay
is supplied by Proposition~\ref{prop:d3-decay}. Apply
Lemma~\ref{lem:global-decay} on each graph patch and then sum over
the finite family of patches to obtain decay in every coordinate.
Proposition~\ref{ltg:thm:main} proves the sufficiency assertion of
Theorem~\ref{thm:global-Lp} for every finite $p>p_c(d,r)$.
The case $p=\infty$ follows from the finiteness of $\sigma_r$.

Proposition~\ref{ltg:thm:main} requires local frequency decay at the
exponent $p$ under consideration. For $\sigma_r$, this decay is
available for every finite $p>p_c(d,r)$, even when $r\ge2(d-1)$.
For such $r$, $M_{\loc}$ is unbounded on $L^2$, including at
$r=2(d-1)$, by Proposition~\ref{prop:global-necessity}.

\subsection{Local-to-global bounds}

We begin with some useful lemmas.

\begin{lemma}\label{ltg:lem:local}
Under \eqref{ltg:eq:hyp}, $M_\loc$ is bounded on $L^p$.
Moreover, if $\widetilde\psi$ is a fixed smooth annular cutoff, then
the same decay, up to a constant, holds with $P_m^{(i)}$ replaced by
the multiplier $\widetilde\psi(2^{-m}\xi_i)$, for $m\ge1$.
\end{lemma}

\begin{proof}
The telescoping identity
\[
 I=\prod_{i=1}^dL_0^{(i)}
 +\sum_{i=1}^d\left(\prod_{\ell<i}L_0^{(\ell)}\right)
 \sum_{m\ge1}P_m^{(i)}
\]
and \eqref{ltg:eq:hyp} bound all terms in the sum by a convergent
geometric series in $m$. For the first term, let
$\Lambda(z)=\prod_i\check\psi_0(z_i)$. Its averaging kernel is
$\int\Lambda(z-\mathrm D_t y)\dd\mu(y)$.
Since $t\in[1,2]^d$ and $y\in[-c,c]^d$, its absolute value is
bounded uniformly in $t$ by $C_N(1+|z|)^{-N}$ for any fixed $N$.
An integrable envelope and Young's inequality give the remaining bound.

For an enlarged annular multiplier, insert the homogeneous partition
\[
 \sum_{\ell\in\Z}\psi(2^{-\ell}\xi_i)=1,
 \qquad \xi_i\ne0.
\]
Only $|\ell-m|\le C_{\widetilde\psi}$ contribute. For $\ell\ge1$,
use \eqref{ltg:eq:hyp} and the uniform multiplier norm of
$\widetilde\psi(2^{-m}\xi_i)$. If $\ell\le0$, then $m$ belongs to a
fixed finite set, and the local $L^p$ bound just proved gives the
desired estimate after changing the constant.
\end{proof}

For $\mathcal I\subset\{1,\ldots,d\}$, let $\pi_{\mathcal I}$ denote
the coordinate projection and let
$\mu_{\mathcal I}=(\pi_{\mathcal I})_\#\mu$ be the corresponding
marginal measure. For a measure $\nu$ on $\R^{|\mathcal I|}$, write
\[
 A_\nu g(x_{\mathcal I},t_{\mathcal I})
 =
 \int_{\R^{|\mathcal I|}}
 g(x_{\mathcal I}-\mathrm D_{t_{\mathcal I}}y_{\mathcal I})
 \dd\nu(y_{\mathcal I}),
 \qquad
 M_\loc^\nu g
 =
 \sup_{t_{\mathcal I}\in[1,2]^{|\mathcal I|}}
 |A_\nu g(\cdot,t_{\mathcal I})|.
\]
We also write
$M^\nu g=\sup_{t_{\mathcal I}\in(0,\infty)^{|\mathcal I|}}
|A_\nu g(\cdot,t_{\mathcal I})|$.
These operators act on functions on $\R^d$ by fixing the remaining
coordinates. When $\mathcal I=\varnothing$, the maximal operator is
interpreted as $g\mapsto\mu(\R^d)|g|$.

\begin{lemma}\label{ltg:lem:marginal}
For every nonempty $\mathcal I$, the measure $\mu_{\mathcal I}$
satisfies \eqref{ltg:eq:hyp} in each of its coordinates, with the same
$p$ and $\epsilon$, and with a constant depending only on the original
data and $d$.
\end{lemma}

\begin{proof}
There is nothing to prove if $\mathcal I=\{1,\ldots,d\}$.
Otherwise write
$x=(u,v)\in\R^{|\mathcal I|}\times\R^h$,
where $h=d-|\mathcal I|$, and choose
$\eta\in C_c^\infty(\R^h)$ equal to one on $[-1,1]^h$.
For $g\in\mathcal S(\R^{|\mathcal I|})$, set
$f_L(u,v)=g(u)\eta(v/L)$.

Assume $\supp\mu\subset[-c,c]^d$. If $L>c$ and
$v\in[-L+c,L-c]^h$, then, with the omitted dilation parameters fixed
equal to one,
\[
 A P_m^{(i)}f_L((u,v),(s,\mathbf 1))
 =
 A_{\mu_{\mathcal I}}P_m^{(i)}g(u,s),
 \qquad
 s\in[1,2]^{|\mathcal I|},\quad i\in\mathcal I.
\]
Here the coordinate labels are arranged according to
$\mathcal I$ and $\mathcal I^c$. Consequently,
\[
 [2(L-c)]^{h/p}
 \left\|M_\loc^{\mu_{\mathcal I}}P_m^{(i)}g\right\|_p
 \le
 C_0 2^{-\epsilon m}
 L^{h/p}\|\eta\|_p\|g\|_p.
\]
Dividing by $L^{h/p}$ and letting $L\to\infty$ proves the claim.
\end{proof}

We will use two standard analytic bounds. Let $H_i$ be the centered
Hardy--Littlewood maximal operator in the $i$th coordinate and put
$\mathcal H=H_1\cdots H_d$. Iterating the one-dimensional maximal
theorem gives
\begin{equation}\label{ltg:eq:HL}
 \|\mathcal Hf\|_p\le C_{d,p}\|f\|_p,
 \qquad 1<p<\infty.
\end{equation}
If $b_\ell(u)=2^{-\ell}(1+2^{-\ell}|u|)^{-2}$, then
\begin{equation}\label{ltg:eq:envelope-HL}
 b_\ell*_i|f|\le CH_if.
\end{equation}
Indeed, split the integral into $|u|\le2^\ell$ and the annuli
$2^{\ell+h-1}<|u|\le2^{\ell+h}$; the resulting bounds form the
geometric sum $C\sum_{h\ge0}2^{-h}H_if$.

For $\ell\in\Z^d$, put $\Delta_\ell=\prod_iP_{\ell_i}^{(i)}$.
We will use the product square-function estimate
\begin{equation}\label{ltg:eq:LP}
 \left\|
 \left(\sum_{\ell\in\Z^d}|\Delta_\ell f|^2\right)^{1/2}
 \right\|_p
 \le C_{d,p}\|f\|_p.
\end{equation}
For completeness, the scalar one-dimensional Littlewood--Paley bound
extends to Hilbert-valued functions by randomizing a finite
orthonormal expansion and applying Khintchine's inequality on both
sides. Applying this Hilbert-valued bound successively in the
coordinates proves \eqref{ltg:eq:LP} by Fubini.

\begin{proof}[Proof of Proposition~\ref{ltg:thm:main}]

Write each positive dilation as $t=2^ks$, where $k\in\Z^d$ and
$s\in[1,2]^d$, and set
\[
 M_kf(x)=\sup_{s\in[1,2]^d}|Af(x,2^ks)|,
 \qquad
 M_Nf=\sup_{|k|_\infty\le N}M_kf,
 \qquad
 B_N=\|M_N\|_{L^p\to L^p}.
\]
Here $2^ks=(2^{k_1}s_1,\ldots,2^{k_d}s_d)$. By a change of variables,
\begin{equation}\label{ltg:eq:rescaling}
 Af(\mathrm D_{2^k}x,2^ks)
 =
 A(f\circ\mathrm D_{2^k})(x,s).
\end{equation}
Lemma~\ref{ltg:lem:local} therefore gives a uniform bound for each
$M_k$, and hence $B_N<\infty$ for every $N$. Each finite maximal
operator is extended to $L^p$ by density before it is used below.
We will obtain a bound independent of $N$.

Let $L_k^{(i)}$ have multiplier $\psi_0(2^{k_i}\xi_i)$, and write
$Q_k^{(i)}=I-L_k^{(i)}$. The algebraic identity
\begin{equation}\label{ltg:eq:inclusion-exclusion}
 I
 =
 \prod_{i=1}^dQ_k^{(i)}
 +
 \sum_{\varnothing\ne\mathcal J\subset\{1,\ldots,d\}}
 (-1)^{|\mathcal J|+1}
 \prod_{i\in\mathcal J}L_k^{(i)}
\end{equation}
separates the term in which every coordinate is high frequency from
all remaining terms. We first estimate the terms containing at least
one low-frequency factor.

\begin{lemma}\label{ltg:lem:low}
For $\varnothing\ne\mathcal J\subset\{1,\ldots,d\}$ and
$\mathcal I=\mathcal J^c$,
\begin{equation}
 \sup_{k\in\Z^d}\sup_{s\in[1,2]^d}
 \left|
 A\left(\prod_{i\in\mathcal J}L_k^{(i)}f\right)(x,2^ks)
 \right|
 \le
 C M^{\mu_{\mathcal I}}
 \left(\prod_{i\in\mathcal J}H_i|f|\right)(x).
\end{equation}
\end{lemma}

\begin{proof}
The kernel of $L_k^{(i)}$ is
$\lambda_{k_i}(u)=2^{-k_i}\check\psi_0(2^{-k_i}u)$.
Assume $\supp\mu\subset[-c,c]^d$. For $|y_i|\le c$,
$t_i=2^{k_i}s_i$, and $s_i\in[1,2]$, Schwartz decay gives
\[
 |\lambda_{k_i}(u-t_iy_i)|
 \le
 C2^{-k_i}(1+2^{-k_i}|u|)^{-2}
 =
 Cb_{k_i}(u).
\]
Expand the convolutions in the $\mathcal J$ coordinates and use the
positivity of $\mu$. The shifted kernels in those coordinates are
bounded by the displayed envelopes, which no longer depend on
$y_{\mathcal J}$. Integrating out these coordinates therefore gives
\begin{align*}
 \left|
 A\left(\prod_{i\in\mathcal J}L_k^{(i)}f\right)(x,t)
 \right|
 &\le
 C A_{\mu_{\mathcal I}}
 \left(
 \prod_{i\in\mathcal J}(b_{k_i}*_i)|f|
 \right)(x,t_{\mathcal I})\\
 &\le
 C A_{\mu_{\mathcal I}}
 \left(
 \prod_{i\in\mathcal J}H_i|f|
 \right)(x,t_{\mathcal I}),
\end{align*}
where \eqref{ltg:eq:envelope-HL} was used in the second step.
Taking the supremum in $k$ and $s$ proves the claim.
\end{proof}

For the application to $\sigma_r$, the projected measures need not
have bounded densities. For $1\le k\le d-1$, let
$\pi_k(y)=(y_1,\ldots,y_{d-k})$ and put
$H(u)=1-\sum_{i=1}^{d-k}|u_i|^r$. The density $w_{r,k}$ of
$(\pi_k)_\#\sigma_r$ satisfies
\[
 w_{r,k}(u)\asymp_{d,r,k}H(u)^{k/r-1},\qquad H(u)>0.
\]
Indeed, apply the coarea formula to
$F(y)=\sum_{i=1}^d|y_i|^r$; the factor $|\nabla F|$ is bounded
above and away from zero on $S_r$, and rescaling the omitted
coordinates by $H(u)^{1/r}$ gives the displayed power.
In particular, omitting two coordinates gives a bounded density
when $r=2$, but an unbounded density near the boundary when $r>2$.
Lemma~\ref{ltg:lem:marginal} transfers the local frequency decay to
these projected measures and provides the input for the dimension
induction below without any bounded-density assumption.

We prove Proposition~\ref{ltg:thm:main} by induction on the dimension,
starting with dimension zero, where the operator is multiplication
by the mass of the measure. For the induction step,
Lemma~\ref{ltg:lem:marginal} shows that every proper coordinate
pushforward $\mu_{\mathcal I}$ satisfies the same hypothesis.
Hence, by the induction hypothesis, $M^{\mu_{\mathcal I}}$ is bounded
on $L^p(\R^{|\mathcal I|})$, and therefore on $L^p(\R^d)$ by Fubini.
Lemma~\ref{ltg:lem:low} and \eqref{ltg:eq:HL} show that all terms in
\eqref{ltg:eq:inclusion-exclusion} other than the all-high term have
total $L^p$ norm at most $C\|f\|_p$, independently of $N$.
No geometric identification of the measures $\mu_{\mathcal I}$ is
needed.

\noindent{\bf High-frequency estimates for \texorpdfstring{$p\ge2$}{p >= 2}.}
For $j\in\N^d$ with every $j_i\ge1$, the relevant relative-frequency
projection is $\Delta_{j-k}$. Indeed,
\[
 \prod_iQ_k^{(i)}f
 =
 \sum_{j\in\N^d}\Delta_{j-k}f,
 \qquad
 |2^{k_i}\xi_i|\asymp2^{j_i}
 \quad\hbox{on the support of }\widehat{\Delta_{j-k}f}.
\]
Choose $\widetilde\psi\in C_c^\infty(\R\setminus\{0\})$ equal
to one on $\supp\psi$, and let $\widetilde\Delta_\ell$ be the
corresponding product multiplier. Put
$|j|_\infty=\max\{j_1,\ldots,j_d\}$ and define
\[
 \mathcal T_{j,k}g=M_k(\widetilde\Delta_{j-k}g).
\]
Rescaling by \eqref{ltg:eq:rescaling}, and applying the enlarged-cutoff
estimate of Lemma~\ref{ltg:lem:local} in a coordinate for which
$j_i=|j|_\infty$, gives
\begin{equation}\label{ltg:eq:block}
 \|\mathcal T_{j,k}g\|_p
 \le
 C2^{-\epsilon|j|_\infty}\|g\|_p.
\end{equation}
The remaining coordinate multipliers have uniformly bounded
$L^p$ operator norms. The constant is uniform in $j$ and $k$.
Since $\widetilde\Delta_{j-k}\Delta_{j-k}=\Delta_{j-k}$,
\begin{equation}
 \|M_k\Delta_{j-k}f\|_p
 \le
 C2^{-\epsilon|j|_\infty}\|\Delta_{j-k}f\|_p.
\end{equation}

For
\[
 T_{j,N}f
 =
 \sup_{|k|_\infty\le N}M_k\Delta_{j-k}f
\]
and $p\ge2$,
\begin{align}
 \|T_{j,N}f\|_p
 &\le
 \left(
 \sum_{|k|_\infty\le N}
 \|M_k\Delta_{j-k}f\|_p^p
 \right)^{1/p}
 \notag\\
 &\le
 C2^{-\epsilon|j|_\infty}
 \left\|
 \left(
 \sum_{|k|_\infty\le N}
 |\Delta_{j-k}f|^p
 \right)^{1/p}
 \right\|_p
 \notag\\
 &\le
 C2^{-\epsilon|j|_\infty}
 \left\|
 \left(
 \sum_{k\in\Z^d}
 |\Delta_{j-k}f|^2
 \right)^{1/2}
 \right\|_p
 \notag\\
 &\le
 C2^{-\epsilon|j|_\infty}\|f\|_p.
\end{align}
The penultimate inequality uses the pointwise inclusion
$\ell^2\subset\ell^p$ for $p\ge2$, while the last inequality follows
from \eqref{ltg:eq:LP}, since $j-k$ runs through $\Z^d$ as $k$ does.

Finally,
\begin{equation}\label{ltg:eq:summability}
 \sum_{\substack{j\in\N^d\\j_i\ge1}}
 2^{-\epsilon|j|_\infty}
 \lesssim_d
 \sum_{m\ge1}m^{d-1}2^{-\epsilon m}
 <\infty.
\end{equation}
Together with the mixed low-frequency bounds, this proves
$B_N\le C$ for $p\ge2$, with $C$ independent of $N$.

\noindent{\bf High-frequency estimates for \texorpdfstring{$1<p<2$}{1 < p < 2}.}
The $\ell^p$ to $\ell^2$ comparison used above is unavailable for
$p<2$. We replace it with an interpolated vector-valued estimate.
All sequences below are indexed by the finite set
$K_N=\{k\in\Z^d:|k|_\infty\le N\}$.

First, \eqref{ltg:eq:block} implies
\begin{equation}\label{ltg:eq:vector-p}
 \left\|\{\mathcal T_{j,k}g_k\}_{k\in K_N}\right\|_{L^p(\ell^p)}
 \le
 C2^{-\epsilon |j|_\infty}
 \left\|\{g_k\}_{k\in K_N}\right\|_{L^p(\ell^p)}.
\end{equation}
Second, each factor of the product kernel of
$\widetilde\Delta_{j-k}$ is dominated by a dilate of
$(1+|u|)^{-2}$. Applying \eqref{ltg:eq:envelope-HL} successively gives
\[
 |\widetilde\Delta_{j-k}g_k|
 \le C\mathcal H(|g_k|)
 \le C\mathcal H G,
 \qquad
 G=\sup_{k\in K_N}|g_k|.
\]
Since $\mu$ is positive,
\[
 \sup_{k\in K_N}\mathcal T_{j,k}g_k
 \le
 C M_N(\mathcal H G).
\]
Hence, by \eqref{ltg:eq:HL},
\begin{equation}\label{ltg:eq:vector-infty}
 \left\|\{\mathcal T_{j,k}g_k\}_{k\in K_N}\right\|_{L^p(\ell^\infty)}
 \le
 C B_N
 \left\|\{g_k\}_{k\in K_N}\right\|_{L^p(\ell^\infty)}.
\end{equation}

To interpolate these estimates, first restrict $s$ to a finite subset
of $[1,2]^d$. For each $k$, choose a measurable function $s_k(x)$
taking values in this set and consider the diagonal linear operator
\[
 (g_k)_k
 \longmapsto
 \left(
 A\widetilde\Delta_{j-k}g_k(x,2^ks_k(x))
 \right)_k.
\]
Its norms at the two endpoints are bounded by
\eqref{ltg:eq:vector-p} and \eqref{ltg:eq:vector-infty}, uniformly in
the choices of $s_k$. Since
\[
 [L^p(\ell^p),L^p(\ell^\infty)]_\theta
 =
 L^p(\ell^2),
 \qquad
 1-\theta=\frac p2,
\]
complex interpolation gives the bound
$C2^{-\epsilon p|j|_\infty/2}B_N^{1-p/2}$.

For a fixed input sequence, choose $s_k(x)$ to maximize the absolute
value in each component. Such choices are measurable for a finite
parameter set. Increasing these finite sets through a countable dense
subset of $[1,2]^d$ and using monotone convergence gives
\begin{equation}\label{ltg:eq:vector-2}
 \left\|\{\mathcal T_{j,k}g_k\}_{k\in K_N}\right\|_{L^p(\ell^2)}
 \le
 C2^{-\epsilon p|j|_\infty/2}B_N^{1-p/2}
 \left\|\{g_k\}_{k\in K_N}\right\|_{L^p(\ell^2)}.
\end{equation}

Apply \eqref{ltg:eq:vector-2} with $g_k=\Delta_{j-k}f$ and use
\eqref{ltg:eq:LP}. Since the supremum over $k$ is bounded pointwise by
the $\ell^2$ norm,
\begin{equation}
 \|T_{j,N}f\|_p
 \le
 C2^{-\epsilon p|j|_\infty/2}
 B_N^{1-p/2}\|f\|_p.
\end{equation}
The decay is summable in $j$ by the same counting as in
\eqref{ltg:eq:summability}. Adding the mixed low-frequency terms gives
\begin{equation}\label{ltg:eq:bootstrap}
 B_N
 \le
 C_1+C_2B_N^\alpha,
 \qquad
 \alpha=1-\frac p2\in(0,1),
\end{equation}
with constants independent of $N$.

Recall that $B_N<\infty$ for every $N$. If $B_N>2C_1$, then
\eqref{ltg:eq:bootstrap} gives
$B_N\le2C_2B_N^\alpha$, and hence
$B_N^{1-\alpha}\le2C_2$. Therefore
\[
 B_N
 \le
 \max\left\{
 2C_1,
 (2C_2)^{1/(1-\alpha)}
 \right\}
\]
uniformly in $N$. This completes the induction step for $1<p<2$.

Finally, $M_Nf(x)\uparrow Mf(x)$ as $N\to\infty$. Hence Fatou's
lemma and the uniform bound for $B_N$ give
$\|Mf\|_p\lesssim\|f\|_p$, completing the proof of
Proposition~\ref{ltg:thm:main}.
\end{proof}

\subsection{Necessity and failure at the critical exponents}
\label{sec:global-necessity}

We now return to the surface measure $\sigma_r$ on $S_r$.
We derive the necessary conditions from $M_{\loc}$ using smooth,
nonnegative test functions. The logarithmic truncations below also
establish failure at the critical exponents.

\begin{proposition}\label{prop:global-necessity}
Let $d\ge3$, $2\le r<\infty$, and $1\le p,q<\infty$. If
\[
 \|M_{\loc}f\|_{L^q(\R^d)}
 \le C\|f\|_{L^p(\R^d)}
 \qquad\text{for every }f\in C_c^\infty(\R^d),
\]
then
\[
 p>p_c(d,r),\qquad p\le q\le dp,\qquad
 \frac{r+d-1}{p}-\frac{d-1}{q}\le d-1.
\]
\end{proposition}

\begin{proof}
Put $n=d-1$. Fix a function $\theta\in C^\infty([0,\infty))$
such that $0\le\theta\le1$, $\theta=1$ on $[0,1]$, and
$\theta=0$ on $[2,\infty)$. Choose $a>0$ sufficiently small,
with $2a<e^{-2}$. All constants below may depend on $d,r,a$,
but are independent of the truncation parameter $\varepsilon$.

\medskip
\noindent\emph{The condition $p>(n+2)/n$.}
Let $p_0=(n+2)/n$, let $e=d^{-1/2}(1,\ldots,1)$, and write
$z_\perp=z-(z\cdot e)e$. Define
\[
 \rho(z)=\bigl(|z_\perp|^4+(z\cdot e)^2\bigr)^{1/4}.
\]
This function is smooth away from the origin. The change of variables
$z_\perp=s u$ and $z\cdot e=s^2 v$ shows that
\[
 |\{z:\rho(z)<s\}|=c_d s^{n+2},\qquad s>0.
\]
For $0<\varepsilon<a/4$, set
\[
 f_\varepsilon(z)=
 \theta\left(\frac{\rho(z)}a\right)
 \theta\left(\frac{\varepsilon}{\rho(z)}\right)
 \frac{\rho(z)^{-n}}{\log(1/\rho(z))}
 \qquad (0<\rho(z)<2a),
\]
and set $f_\varepsilon=0$ elsewhere. Then
$f_\varepsilon\in C_c^\infty(\R^d)$, $f_\varepsilon\ge0$, and
\[
 \sup_{0<\varepsilon<a/4}\|f_\varepsilon\|_{p_0}^{p_0}
 \lesssim
 \int_0^{2a}\frac{s^{n+1-np_0}}{[\log(1/s)]^{p_0}}\dd s
 =\int_0^{2a}\frac{\dd s}{s[\log(1/s)]^{p_0}}
 <\infty.
\]
Since the supports lie in a fixed bounded set, their $L^p$ norms
are also uniformly bounded for every $1\le p\le p_0$.

Let
\[
 x^0=\frac32 d^{-1/r}(1,\ldots,1),
\]
and let $U$ be a sufficiently small bounded open neighborhood of $x^0$
whose compact closure lies in the positive orthant. For $x\in U$, put
\[
 Y(x)=\sum_{j=1}^d x_j,\qquad
 y_j(x)=\left(\frac{x_j}{Y(x)}\right)^{1/r},\qquad
 t_j(x)=\frac{x_j}{y_j(x)}.
\]
Thus $y(x)\in S_r$ and $x=\mathrm D_{t(x)}y(x)$.
At $x=x^0$, all the $t_j(x)$ equal $3/2$; after shrinking $U$,
we have $t(x)\in(1,2)^d$ for every $x\in\overline U$.
A normal vector to $\mathrm D_{t(x)}S_r$ at $x$ has components
\[
 \frac{y_j(x)^{r-1}}{t_j(x)}
 =\frac{y_j(x)^r}{x_j}
 =\frac1{Y(x)}.
\]
Consequently, $x-\mathrm D_{t(x)}S_r$ passes through the origin
with tangent space $e^\perp$, independently of $x$.

The points $y(x)$ stay uniformly away from the coordinate
hyperplanes. A neighborhood of the origin in
$x-\mathrm D_{t(x)}S_r$ can therefore be written as
\[
 z=w+g_x(w)e,\qquad w\in e^\perp,\quad |w|<b,
\]
where $b>0$ is independent of $x\in U$ and
\[
 g_x(0)=0,\qquad \nabla g_x(0)=0,\qquad
 |g_x(w)|\le C|w|^2.
\]
The image of $d\sigma_r$ under $y\mapsto x-\mathrm D_{t(x)}y$,
expressed in these coordinates, has a density bounded below by a
positive constant, uniformly in $x$.
Moreover,
\[
 \rho(w+g_x(w)e)
 =\bigl(|w|^4+g_x(w)^2\bigr)^{1/4}
 \asymp |w|.
\]
Decreasing $b$ if necessary, for all sufficiently small $\varepsilon$
the two cutoffs in $f_\varepsilon$ equal one whenever
$C\varepsilon<|w|<b$. Polar coordinates in
$e^\perp$ now give
\[
 M_{\loc}f_\varepsilon(x)
 \ge A f_\varepsilon(x,t(x))
 \gtrsim
 \int_{C\varepsilon}^{b}\frac{\dd s}{s\log(1/s)}
 \longrightarrow\infty
 \qquad(x\in U),
\]
with a lower bound independent of $x$. As $|U|>0$, this contradicts
the assumed estimate for every $1\le p\le p_0$.

\medskip
\noindent\emph{The condition $p>r/n$.}
When $r\le n+2$, this follows from $p>p_0$. It remains to consider
$r>n+2$. Put $p_1=r/n>1$ and choose a nonnegative
$\eta\in C_c^\infty(\R^n)$ that equals one on $\{|z'|\le a\}$
and is supported in $\{|z'|<2a\}$. Set
\[
 h_\varepsilon(z',z_d)=
 \eta(z')
 \theta\left(\frac{|z_d|}a\right)
 \theta\left(\frac{\varepsilon}{|z_d|}\right)
 \frac{|z_d|^{-n/r}}{\log(1/|z_d|)}
 \qquad(0<|z_d|<2a),
\]
and set $h_\varepsilon=0$ elsewhere. These functions belong to
$C_c^\infty(\R^d)$, are nonnegative, and satisfy
\[
 \sup_{0<\varepsilon<a/4}\|h_\varepsilon\|_{p_1}^{p_1}
 \lesssim
 \int_0^{2a}\frac{\dd s}{s[\log(1/s)]^{p_1}}
 <\infty.
\]
Their $L^p$ norms are consequently uniformly bounded for every
$1\le p\le p_1$.

Consider the fixed set
\[
 V=\{x=(x',x_d):|x'|<a/2,\ 4/3<x_d<5/3\}.
\]
For $x\in V$, choose $t=(1,\ldots,1,x_d)\in[1,2]^d$.
Near the north pole, write
\[
 y=(v,\Phi(v)),\qquad
 \Phi(v)=\left(1-\sum_{j=1}^n|v_j|^r\right)^{1/r}.
\]
For sufficiently small $|v|$, uniformly in $x\in V$,
\[
 x-\mathrm D_ty=(x'-v,x_d(1-\Phi(v))),\qquad
 x_d(1-\Phi(v))\asymp\sum_{j=1}^n|v_j|^r\asymp |v|^r.
\]
The surface density in these coordinates is bounded below by one.
Choose $b_1>0$ small enough that $|x'-v|<a$ and
$0<x_d(1-\Phi(v))<a$ whenever $x\in V$ and $0<|v|<b_1$.
For sufficiently small $\varepsilon$ and
$C\varepsilon^{1/r}<|v|<b_1$, all the cutoffs in
$h_\varepsilon(x-\mathrm D_ty)$ equal one. Hence
\[
 M_{\loc}h_\varepsilon(x)
 \ge A h_\varepsilon(x,t)
 \gtrsim
 \int_{C\varepsilon^{1/r}}^{b_1}
 \frac{\dd s}{s\log(1/s)}
 \longrightarrow\infty
 \qquad(x\in V),
\]
uniformly in $x$. Since $|V|>0$, the assumed estimate fails
for every $1\le p\le p_1$.

\medskip
\noindent\emph{The remaining local improving restrictions.}
The lower bounds in the preceding examples are uniform on fixed sets
of positive measure. Hence they exclude $p\le p_c(d,r)$ for every
finite output exponent $q$.

To see that $p\le q$, take a nonzero, nonnegative function
$g\in C_c^\infty(\R^d)$ and form $N$ sufficiently separated translates.
Their input supports are disjoint, and the corresponding translates of
$M_{\loc}g$ have disjoint supports, since the averaging measures for
$t\in[1,2]^d$ have uniformly bounded support. Positivity gives
$N^{1/q}\|M_{\loc}g\|_q\le C N^{1/p}\|g\|_p$.
Letting $N\to\infty$ yields $p\le q$.

For $q\le dp$, choose $f_\delta\in C_c^\infty(\R^d)$ with
$0\le f_\delta\le1$, equal to one in the $\delta$-neighborhood of
$S_r$ and supported in its $2\delta$-neighborhood. Such functions
are obtained by mollifying indicators of slightly larger neighborhoods.
Since $S_r$ is a compact $C^1$ hypersurface,
$\|f_\delta\|_p\lesssim\delta^{1/p}$. For $|x|<\delta$, symmetry of
$S_r$ gives $Af_\delta(x,\mathbf1)=\sigma_r(S_r)$.
Consequently $\delta^{d/q}\lesssim\delta^{1/p}$, which implies $q\le dp$.

Finally, take a smooth nonnegative function supported in a box with
horizontal widths $O(\delta)$ and vertical width $O(\delta^r)$, equal
to one on a slightly smaller box. For $|x'|<c\delta$ and
$4/3<x_d<5/3$, use $t=(1,\ldots,1,x_d)$ and the north-pole
parametrization $y=(v,\Phi(v))$. On $|v|<c\delta$,
$|x'-v|=O(\delta)$ and $x_d(1-\Phi(v))=O(\delta^r)$.
Choosing the box constants in terms of $d,r$ gives
$M_{\loc}f_\delta(x)\gtrsim\delta^n$ on this set, whose volume is
comparable to $\delta^n$. Since
$\|f_\delta\|_p\lesssim\delta^{(n+r)/p}$, we obtain
\[
 \delta^{n+n/q}\lesssim\delta^{(n+r)/p},
\]
which implies $(n+r)/p-n/q\le n$. This proves the proposition.
\end{proof}

Taking $q=p$ and using $M_{\loc}f\le Mf$, the proposition proves the
necessity in
Theorem~\ref{thm:global-Lp}, including failure at the critical
exponents. The bound on $L^\infty$ follows directly from
$\|Mf\|_\infty\le\sigma_r(S_r)\|f\|_\infty$.
Together with the sufficiency established above, this completes
the proof of Theorem~\ref{thm:global-Lp}.

\appendix
\section{Multiparameter local smoothing in dimension three}
\label{app:d3-smoothing}

This appendix supplies the local smoothing input for dimension three.
We fix $r'=r/(r-1)\in(1,2]$ throughout. All frequency
localizations are contained in signed orthants; no smoothness of the
dual phase across a coordinate hyperplane is used. The quadratic
decoupling theorem and its stability under perturbations are taken from
\cite[Theorems~1.1 and~5.2]{GOZZ2023}. We use the conical projection
argument and the variable-coefficient continuation in
\cite[Section~4]{LLO2025-2}. The normalization below records the
quantitative information needed when the frequency support approaches
the coordinate hyperplanes.

\subsection{Quadratic normal forms}

Fixing the ratio of the first two dilation parameters, and making a
bounded diagonal change of spatial and frequency variables, reduces
the phase to
\begin{equation}\label{eq:app-d3-reduced-phase}
 \Phi(t,s,\theta)=tS^{1/r'},\qquad
 S=s^{r'}+\theta_1^{r'}+\theta_2^{r'}.
\end{equation}
Here $t,s$ lie in a fixed compact subset of $(0,\infty)$ and the signs of
$\theta_1,\theta_2$ have been fixed and reflected to be positive.
We write $z=(t,s)$. The remaining dilation ratio is integrated over a
bounded interval after the estimate for $(t,s)$ has been established.

\begin{lemma}[Quadratic normal form]
\label{lem:d3-quadratic-normal-form}
Put
\[
 H=\diag(\theta_1^{r'-2},\theta_2^{r'-2}),\qquad
 v=(\theta_1^{r'-1},\theta_2^{r'-1})^{\mathsf T},\qquad
 c=(r'-1)S^{1/r'-1}.
\]
Then
\begin{align}
 A:=\partial_\theta^2\partial_t\Phi
   &=c\left(H-\frac{vv^{\mathsf T}}S\right),
       \label{eq:app-d3-hessian-A}\\
 B:=\partial_\theta^2\partial_s\Phi
   &=\frac{ct s^{r'-1}}S
       \left(-(r'-1)H+\frac{(2r'-1)vv^{\mathsf T}}S\right).
       \label{eq:app-d3-hessian-B}
\end{align}
At every such point, the pair of quadratic forms with Hessians $A,B$
is equivalent, by invertible linear changes in the angular and
parameter variables, to
\[
 Q_0(w)=(w_1^2,w_2^2).
\]
If $\theta_i\ge\kappa$ and $|\theta|\lesssim1$, all derivatives through
any prescribed finite order, and the norms of these linear changes
and their inverses, are bounded by a fixed power of $\kappa^{-1}$.
The exponent may depend on $r'$ and the derivative order.
\end{lemma}

\begin{proof}
Differentiation gives \eqref{eq:app-d3-hessian-A} and
\eqref{eq:app-d3-hessian-B}. In particular,
\begin{equation}
 B+\frac{(r'-1)t s^{r'-1}}S A
 =r'(r'-1)t s^{r'-1}S^{1/r'-3}vv^{\mathsf T}.
\end{equation}
Thus $\operatorname{span}\{A,B\}=\operatorname{span}\{H,vv^{\mathsf T}\}$.
The change of angular increments by $H^{-1/2}$ transforms these latter
two matrices into the identity and $ww^{\mathsf T}$, where
\[
 w=(\theta_1^{r'/2},\theta_2^{r'/2})^{\mathsf T}.
\]
A rotation sends $w$ to $|w|e_1$. In these coordinates the determinant of the coefficient
matrix expressing the two transformed Hessians in terms of the radial
and transverse diagonal matrices has absolute value
\begin{equation}\label{eq:app-d3-coefficient-determinant}
 r'(r'-1)^2t s^{r'-1}
 (\theta_1^{r'}+\theta_2^{r'})S^{2/r'-4}>0.
\end{equation}
An invertible linear combination of the parameter directions therefore
gives the two coordinate squares. The factor $1/2$ converting Hessians
into quadratic forms is included in this last change of variables.

On the indicated compact parameter set, $S$ is bounded above and below.
The entries of $H^{\pm1/2}$ have polynomial bounds in $\kappa^{-1}$,
and the determinant in \eqref{eq:app-d3-coefficient-determinant} is
bounded below by $c_{r'}\kappa^{r'}$. Matrix inversion and differentiation
give the asserted polynomial bounds. The rotation here acts only on
the frozen angular increments; it does not assert rotational symmetry
of \eqref{eq:app-d3-reduced-phase}.
\end{proof}

\subsection{Uniform two-parameter decoupling}

We first prove a variable-coefficient decoupling estimate uniform
over the following class of phases. Angular partitions below consist of smooth,
finitely overlapping squares of the stated side length. They leave
the radial variable $\rho\in[1/2,2]$ undecomposed. Norms on a bounded
output set may equivalently be formulated with a rapidly decreasing
weight on a slightly larger set, as in
\cite[Sections~2.3--2.7]{BHS2020}.

\begin{lemma}[Uniform two-parameter decoupling]
\label{lem:d3-uniform-decoupling}
For every $\varepsilon>0$ and remainder order $M$ there are a finite
integer $N$ and a constant $\epsilon_0>0$ with the following property.
Let $\Psi(z,\theta)$ be defined on fixed boxes in
$\R^2_z\times\R^2_\theta$, with fixed support margins, and suppose that
\[
 \Psi(0,\theta)=0,\qquad
 \Psi(z,0)=\partial_\theta\Psi(z,0)=0.
\]
Writing $E(z,\theta)=\Psi(z,\theta)-z\cdot Q_0(\theta)$, assume
\begin{equation}\label{eq:app-d3-normalized-class}
 \sup|\partial_z^\alpha\partial_\theta^\gamma E|
 \le\epsilon_0
 \quad
 \left(1\le|\alpha|,\ 2\le|\gamma|,
              \ |\alpha|+|\gamma|\le N\right).
\end{equation}
For a smooth amplitude supported in a fixed output box, with bounded
derivatives through order $N$, set
\[
 \mathcal E_\lambda^\Psi g(x,z)
 =a(x,z)\iint
 e^{i\lambda\rho(x'\cdot\theta+x_3+\Psi(z,\theta))}
 g(\theta,\rho)\dd\theta\dd\rho.
\]
If $\lambda\ge1$ and $g_\nu$ denotes the angular localization of $g$
to a square of side $\lambda^{-1/2}$, then
\begin{equation}\label{eq:app-d3-variable-decoupling}
 \|\mathcal E_\lambda^\Psi g\|_6
 \lesssim_{\varepsilon,M}
 \lambda^{1/3+\varepsilon}
 \left(\sum_\nu\|\mathcal E_\lambda^\Psi g_\nu\|_6^6\right)^{1/6}
 +\lambda^{-M}\|g\|_2.
\end{equation}
The right hand side is understood with the customary enlarged smooth
output cutoff, or the corresponding weighted norms. The constants are
uniform over the class \eqref{eq:app-d3-normalized-class}. Smooth
frequency-dependent amplitudes with the same finite derivative bounds
are also permitted.
\end{lemma}

\begin{proof}
For $Q_0$, the quadratic decoupling exponent at $6$ is $1/3$.
Theorems~1.1 and~5.2 of \cite{GOZZ2023}, together with the conical
projection argument of \cite[Section~4.2]{LLO2025-2}, give
uniform $\ell^6$ decoupling with exponent $1/3+\varepsilon/4$ for the
constant-coefficient conic models in a fixed neighborhood of $Q_0$.
We explain the continuation to \eqref{eq:app-d3-normalized-class}; this
also identifies the uniformity required below.

Fix a small number $\eta>0$, to be chosen in terms of $\varepsilon$.
As in the weighted induction in \cite[Section~2.7]{BHS2020}, keep both
the oscillation scale $\lambda$ and an output radius $R$, with
$1\le R\le\lambda^{1-\eta}$, until the last step. In coordinates
scaled by $\lambda$, a ball of radius $R$ corresponds to an output
ball of radius $R/\lambda$. Fix a large number $K$, independent of
$R,\lambda$. On output boxes of side $K/\lambda$, expand in the two
parameter variables about the center $z_0$. The term $\Psi(z_0,\theta)$ is absorbed into the arbitrary
frequency-side input $g$. Constant and linear angular terms in
$\partial_z\Psi(z_0,\theta)$ are removed by translations of the spatial
variables. The remaining phase, in output coordinates of size $K$,
is the conic quadratic model plus an error whose finitely many
frequency derivatives are bounded by
\begin{equation}\label{eq:app-d3-finite-scale-error}
 C_N\bigl(K\epsilon_0+K^2/\lambda\bigr).
\end{equation}
Indeed, the first term follows by integrating
\eqref{eq:app-d3-normalized-class} twice in $\theta$; the second is the
Taylor remainder in $z$. Choose $\epsilon_0\le c_NK^{-1}$ and restrict
first to $\lambda\ge C_NK^2$. Fourier series expansion of the exponential
of this error, in the genuine frequency coordinates
$\omega=\rho(\theta,1)$, compares the two operators with uniformly
summable coefficients. The modulations arising from this expansion
are spatial translations. Applying the comparison in both directions
transfers the conic decoupling inequality to these output boxes.
Summing sixth powers over a bounded-overlap cover gives coarse
decoupling into angular squares of side $K^{-1/2}$, at a cost
\begin{equation}\label{eq:app-d3-coarse-cost}
 C_\varepsilon K^{1/3+\varepsilon/4}.
\end{equation}
For the weighted comparison, first localize $z$ to a smooth enlarged
coarse parameter box. We use the constant-coefficient conic inequality
with this same parameter localization and the permitted enlarged
spatial cutoff. The error in \eqref{eq:app-d3-finite-scale-error}
depends only on $z$ and the genuine frequency variables $\omega$,
and its finite-order derivatives are uniformly bounded on this box.
The Fourier coefficients $c_m(z),d_m(z)$ of its exponential and inverse
exponential therefore satisfy, for any prescribed weight order $L$,
\[
 \sum_{m\in\Z^3}(1+|m|)^L
 \sup_z\bigl(|c_m(z)|+|d_m(z)|\bigr)\le C_N,
\]
provided $N$ is sufficiently large. In the scaled output coordinates,
the resulting modulations translate only the spatial variables; the
parameter cutoff is unchanged. For the rapidly decreasing output
weights, we have
$w_B(x-m,z)\lesssim_L(1+|m|)^Lw_B(x,z)$.
Thus both comparisons hold in the weighted norms, with uniformly
summable coefficients. The weighted comparison controls the
$K\epsilon_0$ term uniformly in $\lambda$. The remainder estimates
used in the induction are established below.

Choose the weights $w_b$ of the coarse cover of an output ball $B_R$
so that $\sum_b w_b\lesssim w_{B_R}$, using the same enlarged
parameter cutoffs. Then
\[
 \sum_b\|F\|_{L^6(w_b)}^6\lesssim\|F\|_{L^6(w_{B_R})}^6.
\]
Applying this to each cap contribution and interchanging the two
finite sums combines the coarse right hand sides without a factor
equal to the number of boxes. The remaining Taylor and Fourier-series
tails are handled by the weighted approximation argument in
\cite[Sections~2.6--2.7]{BHS2020}. It uses Taylor expansion and
integration by parts in the present five output variables, and its
remainder is retained in the two-scale form below until the last step.

For completeness, closure under the rescaling used to iterate this
step is particularly simple here. If $\theta_\nu$ is the center of
a $K^{-1/2}$ square, put
\begin{equation}
\begin{split}
 \Psi_\nu(z,u)=K\bigl(&\Psi(z,\theta_\nu+K^{-1/2}u)
                       -\Psi(z,\theta_\nu)\\
                     &-K^{-1/2}\partial_\theta\Psi(z,\theta_\nu)
                                      \cdot u\bigr).
\end{split}
\label{eq:app-d3-exact-rescaling}
\end{equation}
The angular-affine terms are removed for every $z$.
Consequently $\Psi_\nu$ satisfies the same three vanishing conditions.
For $|\gamma|\ge2$,
\[
 \partial_z^\alpha\partial_u^\gamma
   \bigl(\Psi_\nu-z\cdot Q_0(u)\bigr)
 =K^{1-|\gamma|/2}
   (\partial_z^\alpha\partial_\theta^\gamma E)
          (z,\theta_\nu+K^{-1/2}u).
\]
Thus the class \eqref{eq:app-d3-normalized-class} is preserved, with
no accumulating enlargement of $\epsilon_0$. Recentering a parameter
box also preserves the class after subtracting its value at the new
center into the frequency-side input and again removing the full
angular-affine part. The exact spatial changes
are
\[
 y'=K^{1/2}\bigl(x'+\partial_\theta\Psi(z,\theta_\nu)\bigr),\qquad
 y_3=K\bigl(x_3+x'\cdot\theta_\nu+\Psi(z,\theta_\nu)\bigr).
\]
They turn the cap contribution into an operator at frequency
$\lambda/K$ with phase $\Psi_\nu$. Their spatial Jacobian is constant;
it cancels between the two sides of decoupling. The transformed output
sets are covered with bounded overlap and summed in sixth power, so
their volume does not cause a loss in $K$. This is the parabolic
rescaling step of \cite[Section~2.5]{BHS2020} in the present
spatially translation-invariant setting.

We now carry out the induction, including the remainder estimates.
Let
$D(\lambda;R)$ be the smallest uniform constant in decoupling on an
output ball of radius $R/\lambda$, with main coefficient
$D(\lambda;R)R^{1/3+\varepsilon/2}$, angular scale $R^{-1/2}$, and
remainder
\begin{equation}\label{eq:app-d3-weighted-remainder}
 R^{C_0}(\lambda/R)^{-c\eta N}\|g\|_2.
\end{equation}
The fixed integer $C_0$ is chosen larger than the dimension-dependent
powers from the coverings and Jacobians, and $c>0$ is a fixed small
constant. The weighted approximation argument just described supplies
this remainder: the Fourier-series tails are integrated by parts
a fixed fraction of $N$ times, and the finitely many amplitude factors
depending on the center of the small output box are expanded to a
comparable order on a region
of diameter $R/\lambda$. This is exactly the role of the two scales
in \cite[Sections~2.6--2.7]{BHS2020}. The calculation uses no
one-parameter curvature theorem.

After \eqref{eq:app-d3-exact-rescaling}, cover the transformed ball by
balls of radius $R/K$ in the coordinates scaled by $\lambda/K$.
Taking sixth powers before summing ensures bounded overlap, without a
factor equal to the number of balls. Both scales change by the same
factor:
\[
 (\lambda,R)\longmapsto(\lambda/K,R/K),\qquad
 \frac{\lambda/K}{R/K}=\frac\lambda R.
\]
In particular, the small factor in
\eqref{eq:app-d3-weighted-remainder} is preserved. Under this
rescaling, the angular substitution has Jacobian $K^{-1}$ and the
spatial change has Jacobian $K^2$. Thus, on returning to the original
variables, the output $L^6$ norm is multiplied by $K^{-4/3}$,
while the rescaled
input $\widetilde g_\nu(u,\rho)=g_\nu(\theta_\nu+K^{-1/2}u,\rho)$
satisfies $\|\widetilde g_\nu\|_2=K^{1/2}\|g_\nu\|_2$.
The transformed ball is covered by $O(K^2)$ balls of the new radius.
Even using the triangle inequality for these remainders gives
\[
 K^{7/6}(R/K)^{C_0}(\lambda/R)^{-c\eta N}\|g_\nu\|_2.
\]
Including \eqref{eq:app-d3-coarse-cost}, the factor relative to
\eqref{eq:app-d3-weighted-remainder} is
$C_\varepsilon K^{3/2+\varepsilon/4-C_0}$.
Choose $C_0$ sufficiently large, and use
$(\sum_\nu\|g_\nu\|_2^6)^{1/6}\lesssim\|g\|_2$.
The resulting geometric sum is absorbed into the remainder in
\eqref{eq:app-d3-weighted-remainder}. Summing the main terms in sixth
power and using bounded overlap gives
\[
 D(\lambda;R)
 \le C_\varepsilon K^{-\varepsilon/4}
       D(\lambda/K;R/K)+C,
\]
where the constant term accounts for the remainder and contributions
from bounded scales. This induction is used for $R\ge C_NK^4$; smaller
$R$ are covered by trivial finite decoupling. Its hypotheses are
preserved, since $R\le\lambda^{1-\eta}$ implies
$R/K\le(\lambda/K)^{1-\eta}$. First choose $K$ so that
$C_\varepsilon K^{-\varepsilon/4}\le1/2$, and then choose
$\epsilon_0\le c_NK^{-1}$. Induction in $R$ yields a uniform bound
for $D(\lambda;R)$.

Finally take $R=\lambda^{1-\eta}$ and cover the original bounded
output set by balls of radius $R/\lambda$. The main terms sum in
sixth power with bounded overlap. Each angular square of side
$R^{-1/2}$ contains $O(\lambda/R)$ squares of side
$\lambda^{-1/2}$. The triangle inequality followed by H\"older's
inequality therefore refines the former partition to the latter at
cost at most $(\lambda/R)^{5/6}$. The total main exponent is
\[
 (1-\eta)(1/3+\varepsilon/2)+5\eta/6
 \le 1/3+\varepsilon
\]
if $\eta$ is sufficiently small in terms of $\varepsilon$.
The number of output balls is $O((\lambda/R)^5)$. Even summing
the remainders by the triangle inequality, their total is at most
\[
 C(\lambda/R)^5 R^{C_0}(\lambda/R)^{-c\eta N}\|g\|_2
 =C\lambda^{(1-\eta)C_0+5\eta-c\eta^2N}\|g\|_2.
\]
Choose $N$ so that $c\eta^2N\ge M+(1-\eta)C_0+5\eta$.
This gives the required $C\lambda^{-M}\|g\|_2$ remainder and proves
\eqref{eq:app-d3-variable-decoupling}. Smooth frequency-dependent
amplitudes are handled by the same finite Fourier expansions.

The comparison is performed at the fixed scale $K$. The exact
rescaling \eqref{eq:app-d3-exact-rescaling} then extends the estimate
to all frequencies $\lambda\ge1$.
\end{proof}

\subsection{Conic smoothing away from the coordinate axes}

\begin{proposition}[Conic smoothing with polynomial constants]
\label{prop:d3-conic-smoothing}
Let $\mathcal T$ be a fixed compact subset of $(0,\infty)^3$, and put
\[
 h_{r'}(\xi)=\bigl(|\xi_1|^{r'}+|\xi_2|^{r'}+|\xi_3|^{r'}\bigr)^{1/r'}.
\]
Let $0<\kappa\le1/4$. Suppose $b_{\lambda,\kappa}(t,\xi)$ is supported where
$|\xi|\sim\lambda$ and $|\xi_i|\ge\kappa|\xi|$ for $i=1,2,3$.
Assume its order-zero symbol seminorms through a sufficiently large
finite order, including derivatives in $t$, are bounded by a fixed
power of $\kappa^{-1}$, uniformly in $\lambda\ge1$. More precisely,
we assume $\lambda^{|\alpha|}|\partial_\xi^\alpha\partial_t^\beta
b_{\lambda,\kappa}|\le C\kappa^{-A_0}$ up to that order, on a fixed
parameter neighborhood of $\mathcal T$. Then, for either choice of sign and every
$\varepsilon>0$, there is a finite $A=A(r,\varepsilon,b)$ such that
\begin{equation}\label{eq:app-d3-smoothing}
 \left\|\int e^{i(x\cdot\xi\,\pm h_{r'}(\mathrm D_t\xi))}
       b_{\lambda,\kappa}(t,\xi)\widehat f(\xi)\dd\xi
 \right\|_{L^6(\R^3\times\mathcal T)}
 \lesssim_{r,\varepsilon,b}
 \kappa^{-A}\lambda^{1/3+\varepsilon}\|f\|_6.
\end{equation}
The assertion also applies to the finitely many order-zero amplitudes
produced by normalized differentiation in $t$.
\end{proposition}

\begin{proof}
Permuting coordinates and fixing signs reduces to
$\xi=\lambda\rho(\theta,1)$, with $\rho\sim1$ and
$\theta_i\gtrsim\kappa$. Fixing the ratio $t_1/t_2$ gives
\eqref{eq:app-d3-reduced-phase}. The relevant change of variables and
its inverse have bounded Jacobians, so it suffices to prove the
estimate in the two remaining parameters, uniformly in that ratio.

We give the finite normalization which is responsible for the
polynomial dependence on $\kappa$. At a point $(z_0,\theta_0)$, choose
matrices $L,C$ from Lemma~\ref{lem:d3-quadratic-normal-form}, including
the factors converting Hessians to quadratic forms. Set
\[
 \theta=\theta_0+\rho_0Lw,\qquad
 z=z_0+\tau_0C\zeta,\qquad
 \rho_0=\tau_0=c_0\kappa^B,
\]
where $c_0>0$ is fixed sufficiently small and $B$ sufficiently large. Write
$F(z,\theta)=\Phi(z,\theta)-\Phi(z_0,\theta)$ and define
\[
 \mathcal R(z,\theta)=F(z,\theta)-F(z,\theta_0)
            -\partial_\theta F(z,\theta_0)\cdot(\theta-\theta_0).
\]
Taylor's integral formula and the polynomial bounds in
Lemma~\ref{lem:d3-quadratic-normal-form} give
\begin{equation}
 \left\|
 \frac{\mathcal R(z_0+\tau_0C\zeta,\theta_0+\rho_0Lw)}
      {\tau_0\rho_0^2}
   -\zeta\cdot Q_0(w)
 \right\|_{C^N}
 \lesssim_{r,N}\kappa^{-C_N}(\rho_0+\tau_0).
\end{equation}
The nonlinear-parameter error before division is
$O(\tau_0^2\rho_0^2)$: the full angular-affine Taylor polynomial has
been removed for every parameter value. Choosing $B$ after $N$ makes
the right hand side smaller than the fixed $\epsilon_0$ in
Lemma~\ref{lem:d3-uniform-decoupling}. The normalized phase satisfies
the three vanishing conditions of that lemma.

Apply the lemma at effective frequency
$\lambda' =\lambda\tau_0\rho_0^2$. The phase
$\Phi(z_0,\theta)$ is absorbed into the frequency-side input of the
decoupling inequality. The $L^6$ norms of the cap contributions are
estimated below using the original phase. The angular-affine terms
are spatial translations
depending on the parameter, with joint Jacobian one. The numbers of
boxes in the cover, all coordinate Jacobians and amplitude seminorms,
and the refinements of the pulled-back angular caps cost only fixed
powers of $\kappa^{-1}$. In particular, caps of width
$(\lambda')^{-1/2}$ pull back to angular sets of width comparable to
$\lambda^{-1/2}\tau_0^{-1/2}L$; refinement to ordinary
$\lambda^{-1/2}$ caps costs a polynomial factor. If $\lambda'<1$,
then $\lambda$ itself is bounded by a power of $\kappa^{-1}$, and the
elementary compact-frequency multiplier estimate suffices.
This gives decoupling for the original phase with coefficient
$C_{r,\varepsilon}\kappa^{-A}\lambda^{1/3+\varepsilon}$.

We finish by passing from decoupling to smoothing. On an angular cap
of width $\lambda^{-1/2}$, subtract the tangent linear function of
$\pm h_{r'}(D_t\xi)$. After the usual cap rescaling, the residual
phase and the amplitude have finitely many derivatives bounded by a
power of $\kappa^{-1}$. Integration by parts in the rescaled frequency
variables therefore bounds the convolution kernel in $L^1$ by
$C\kappa^{-A}$. This estimate is applied to the original phase,
after decoupling, and is uniform in the parameters. Also,
\[
 \left(\sum_\nu\|f_\nu\|_6^6\right)^{1/6}\lesssim\|f\|_6,
\]
by interpolation between almost orthogonality on $L^2$ and the
uniform cap-kernel bound on $L^\infty$. These two observations yield
\eqref{eq:app-d3-smoothing} with a compact spatial cutoff.

To remove that cutoff, partition physical space into unit cubes.
The kernel is rapidly decreasing outside a fixed polynomially
enlarged neighborhood of each cube, since the phase gradients and
the required symbol seminorms have polynomial bounds in
$\kappa^{-1}$. Localizing $f$ to those neighborhoods and summing
sixth powers gives the global spatial norm, with another polynomial
loss. The decoupling remainder is harmless: on each such neighborhood
$\|f\|_2\lesssim\kappa^{-A}\|f\|_6$, and the frequency-side
$L^2$ norm differs by at most a fixed power of $\lambda$.
The remainder order in Lemma~\ref{lem:d3-uniform-decoupling} is chosen
larger than this power. Finally integrate over the fixed dilation
ratio. All arguments use only finitely many symbol derivatives, so
they also apply to the stated differentiated amplitudes.
\end{proof}

\subsection{From local smoothing to annular maximal estimates}

We prove Proposition~\ref{prop:d3-annular-smoothing}. We work with
the original physical annulus, rather than the horizontally normalized
operator. Fix a signed graph patch and put
$\delta_*=\min(\delta_1,\delta_2)$. All constants below may depend
on the fixed graph cutoff and on $r$; none depends on $R$ or $\delta$.
The argument uses finitely many derivatives of the rescaled amplitudes
$B_\delta$, whose uniform bounds were established in the proof of
Proposition~\ref{prop:dyadic-patch-bounds}.

First suppose that $R\delta_*^r\ge1$, so both horizontal coordinates
are active. In the stationary sectors of
Lemma~\ref{lem:local-exact-sector}, returning to physical frequency
coordinates gives
\[
 |\xi_3|\asymp R,\qquad
 |\xi_i|\asymp R\delta_i^{r-1}\quad(i=1,2).
\]
In particular $|\xi|\asymp R$, and these sectors lie in a union of
signed cones with $|\xi_j|\ge c\delta_*^{r-1}|\xi|$ for all $j$.
The constants in these inclusions are uniform for $t\in[1,2]^3$;
we may therefore use enlarged frequency cutoffs independent of $t$.
The stationary part of the averaging multiplier is a finite sum of
terms
\begin{equation}
 R^{-1} e^{\pm i h_{r'}(\mathrm D_t\xi)} b_{\delta,R}(t,\xi),
 \qquad |\xi|\asymp R.
\end{equation}
For every fixed finite number of derivatives, the corresponding
order-zero symbol seminorms of $b_{\delta,R}$ are at most a fixed
power of $\delta_*^{-1}$. Indeed, the two stationary-phase factors
are $V_\delta\prod_{i=1}^2(R\delta_i^r)^{-1/2}$; after extracting
$R^{-1}$, this and the changes between the sector and physical
frequency coordinates have only polynomial dependence on
$\delta_*^{-1}$. The same is true of the parameter derivatives
required below. This follows either directly from the exact
factorization or from stationary phase on a fixed signed annulus.

Write $A^{\mathbf m}_{\mathrm{stat},R}f$ for this stationary part.
For every $\beta\in\{0,1\}^3$, differentiating its multiplier shows
that $R^{1-|\beta|}\partial_t^\beta
A^{\mathbf m}_{\mathrm{stat},R}$ has the phase and amplitude covered
by Proposition~\ref{prop:d3-conic-smoothing}. Applying that proposition
with $\kappa=c\delta_*^{r-1}$ gives
\begin{equation}
 \big\|R^{-|\beta|}\partial_t^\beta
 A^{\mathbf m}_{\mathrm{stat},R}f\big\|_{L^6_{x,t}}
 \lesssim_{r,\varepsilon}
 \delta_*^{-A}R^{-2/3+\varepsilon}\|f\|_6.
\end{equation}
Here and below we increase the finite exponent $A$ when necessary.
The scaled parameter Sobolev inequality
\cite[Lemma~2.1]{LLO2025-2}, with each parameter scale equal to $R$,
then yields
\begin{align*}
 \left\|\sup_{t\in[1,2]^3}
 |A^{\mathbf m}_{\mathrm{stat},R}f(\cdot,t)|\right\|_6
 &\lesssim R^{1/2}
 \sum_{\beta\in\{0,1\}^3}
 \big\|R^{-|\beta|}\partial_t^\beta
 A^{\mathbf m}_{\mathrm{stat},R}f\big\|_{L^6_{x,t}}\\
 &\lesssim \delta_*^{-A}R^{-1/6+\varepsilon}\|f\|_6.
\end{align*}

We also record why the nonstationary sectors are harmless for this
\emph{vertical}-frequency localization. In the present regime
$I=\{1,2\}$. Every sector other than the stationary ones has either
an index in $I\setminus S$ with $k_i=0$, or an index with $k_i\ge1$.
The exact factorization supplies, for any prescribed $K$, respectively
the factors $\lambda_i^{-K}$ or $(2^{k_i}\lambda_i)^{-K}$.
For a multiplier supported in a fixed rescaled frequency box, its
kernel $L^1$ norm is bounded by finitely many derivatives of that
rescaled multiplier. The same estimate applies after taking the
at most three parameter derivatives needed for the unscaled
Sobolev inequality. Differentiating the residual oscillatory factor
costs at most $R^{C_0}\prod_{i:k_i\ge1}2^{C_0k_i}$, where the finite
$C_0$ is fixed by this derivative requirement. Amplitude derivatives
are controlled by Lemma~\ref{lem:local-exact-sector}.
Using one of the nonstationary factors to gain a power of $R$,
and $\lambda_i=R\delta_i^r$, consequently gives
\[
 \left\|\sup_{t\in[1,2]^3}
 |A^{\mathbf m}_{S,k}P_R^{(3)}f(\cdot,t)|\right\|_6
 \lesssim_K
 \delta_*^{-C_K}R^{C_0-K}
 \prod_{i:k_i\ge1}2^{-(K-C_0)k_i}\|f\|_6
\]
for every nonstationary sector. The unused stationary factors are
bounded by one. Choosing $K>C_0+N+2$ and summing the remote indices
therefore bounds their total contribution by
$C_N\delta_*^{-C_N}R^{-N}\|f\|_6$. This argument includes arbitrarily
large horizontal frequencies; no replacement of $P_R^{(3)}$ by an
isotropic frequency projection is being made.

Finally, if $R\delta_*^r<1$, the elementary $L^1$ and $L^\infty$
maximal estimates in Proposition~\ref{prop:dyadic-patch-bounds} give
\[
 \|M_{\loc}^{\mathbf m}P_R^{(3)}f\|_6
 \lesssim R^{1/6}V_\delta\|f\|_6
 \le \delta_*^{-r/3}R^{-1/6}\|f\|_6.
\]
Thus the same conclusion holds in the inactive regime after enlarging
$A$. The finite sums over signs and graph charts preserve the bound.
This proves Proposition~\ref{prop:d3-annular-smoothing}.

\section*{Acknowledgments}
This work is supported by the National Research Foundation of Korea
(Grant No.~RS-2024-00461749). During the preparation of this manuscript, the authors used
OpenAI's Codex to assist with English-language editing, LaTeX
preparation, and the review of mathematical arguments.
The authors take full responsibility for the mathematical results
and the final content of the paper.

\end{document}